\documentclass{article}
\usepackage[margin=1in]{geometry}

\usepackage{amsmath,amssymb,amsthm,mathtools}
\usepackage{braket}

\usepackage{hyperref}

\usepackage{graphicx}
\usepackage{subcaption}

\usepackage{caption}
\newtheorem{theorem}{Theorem}
\newtheorem{definition}{Definition}
\newtheorem{proposition}{Proposition}
\newtheorem{lemma}{Lemma}
\newtheorem{corollary}{Corollary}
\theoremstyle{remark}

\begin{document}
\title{What is Smoothness?}
\author{Zachary P. Bradshaw\footnote{Corresponding author: zak@qodexquantum.ai}}
\date{QodeX Quantum}

\maketitle

\begin{abstract}
Smoothness of a function on the real line is reflected in the decay of its Fourier transform, which suggests that smoothness of a function in $L^2(G)$ for a group $G$ should mean concentration of the Fourier coefficients at low frequency. Such a reading presupposes an ordering of the irreducible representations of $G$, but for non-abelian $G$, no ordering is canonical. Given a symmetric generating set $S$, the Laplacian of the associated Cayley graph is block diagonal over the dual, and we order the irreps by the mean of the eigenvalues in each block. This produces an ordering function $\omega:\widehat{G}\to\mathbb{R}$ that depends only on the pair $(G,S)$. This function is bounded between zero and two, vanishing only at the trivial representation and achieving the upper bound exactly when the Cayley graph is bipartite. We then ask how much freedom the construction has. Within the class of operators that are Hermitian, left invariant, annihilate the constant functions, and assign an energy insensitive to complex conjugation, the induced orderings are exactly the real functions on the dual vanishing at the trivial representation and agreeing on conjugate pairs, and the orderings coming from inversion orbits of conjugacy classes form a basis for them. We cut the freedom down further by requiring two additional inputs: nonnegativity of the class weights and a declaration of which group elements count as uniform incremental changes, which pins the operator to the Cayley-Laplacian up to positive scale. We observe that the construction persists for compact groups even though the Cayley graph does not, and we extend the theory to finite sets carrying a transitive group action, where the acting group selects which frequencies exist and the generating set orders them. The answer to the title question is therefore that smoothness is a property of a function together with a choice of group and generating set, not of the function alone.
\end{abstract}


\section{Introduction}

The canonical notion of smoothness taught in a single variable calculus course is given by a function which is continuously differentiable to any order. By applying integration by parts repeatedly, it can be shown that a smooth function with derivatives in $L^1(\mathbb{R})$ for any order $m$ has a Fourier transform that decays superpolynomially. However, the converse is false; the decay of the Fourier transform does not fully characterize the space of smooth functions. The Fourier transform does, however, preserve the closely related Schwartz space. Indeed, denoting the Schwartz space by $\mathcal{S}(\mathbb{R})$, one finds that $f\in\mathcal{S}(\mathbb{R})$ if and only if $\widehat{f}\in\mathcal{S}(\mathbb{R})$, making Schwartz functions the natural function space to do Fourier analysis.

This result motivates the perspective of Belis et al.~\cite{belis2026spectral} that smoothness on $L^2(G)$ for an arbitrary group $G$ should be interpreted as a bias in which higher order frequencies rapidly decay~\cite{rahaman2019spectral,xu2019frequency}. In their work, they construct a generalized set of ``frequencies'' from the irreducible representations (irreps) of $G$ which can be low-pass filtered to induce smoothness on a data distribution. For abelian groups, such as $\mathbb{Z}_2^n$, a natural ordering of the irreps is available in the form of the Hamming weight; the irreps are labeled by the elements of $\mathbb{Z}_2^n$, and the elements which contain more zeros are interpreted as smoother than those with fewer. This ordering makes great sense when one Fourier expands a function $f:\mathbb{Z}_2^n\to\mathbb{C}$, producing
\begin{equation}\label{eq:z2n-fourier}
    f(x)=\sum_{k\in\mathbb{Z}_2^n}\widehat{f}(k)(-1)^{k\cdot x}.
\end{equation}
The dependence on $x$ shows up entirely in the sign function $(-1)^{k\cdot x}$, and when a component $k_i$ of $k$ is zero, the corresponding component $x_i$ of $x$ does not affect the value of $f$. Thus, those $k$ with high Hamming weight have the most potential to change $f$, validating the perspective that such $k$ should be interpreted as higher frequency.

This perspective is perfectly natural, but non-abelian groups introduce an added complexity: the ordering of the irreps is no longer obvious, and the problem is exacerbated by the fact that non-abelian groups have irreps of dimension greater than one. In this work, we show that such an ordering does exist, and the mathematics used to obtain it has existed for decades in the works of Diaconis~\cite{diaconis1988group}, Babai~\cite{babai1979spectra}, and Terras~\cite{terras1999fourier}. Already, an ordering has been constructed for the irreps of $S_n$ in the context of random walks, where the object of interest is not the ordering itself, but the spectral gap question of which nontrivial irrep sits lowest~\cite{diaconis1981generating,caputo2010proof}. The same eigenvalues appear in the study of fitness landscapes, where the energy of a function across the Laplacian eigenspaces of a Cayley graph is called its amplitude spectrum and is read as a measure of ruggedness~\cite{stadler1996landscapes,hordijk1998amplitude,rockmore2002fast}, and in graph signal processing on Cayley graphs, where the block structure over the dual is used to select a Fourier basis or a frame rather than to compare one generating set against another~\cite{ghandehari2019non,chen2021signal,beck2024frames}. The same spectral data organizes stationary kernels on groups and their homogeneous spaces~\cite{azangulov2024stationary}. Here we use this pre-existing spectral machinery to extract an irrep-level ordering for arbitrary finite groups, followed by an extension to compact groups, using tools from graph signal processing, in which frequencies are characterized by Laplacian spectra~\cite{ortega2018graph,leus2023graph,dong2020graph}. 

Crucially, our notion of smoothness depends on the choice of generating set for the group of interest. This choice serves as a declaration of which elements of the group count as a small perturbation, and this structure is exhibited by the corresponding Cayley graph, in which group elements that differ by few elements of the generating set are closer together. The quadratic form of the Laplacian of this graph quantifies how much a given function changes when its argument is replaced by the group elements adjacent to it, thereby encoding a natural notion of smoothness. For this reason, we block diagonalize the Laplacian and associate a scalar to each block and therefore each irrep, and this scalar encodes the ordering of the irreps.

Concretely, we define the ordering parameter of an irrep $\sigma$ to be the mean of the eigenvalues of the Laplacian on the $\sigma$-isotypic Peter-Weyl block $\mathcal{A}_\sigma$, which takes the form $\omega(\sigma)=1-\operatorname{Tr}[M_\sigma]/d_\sigma$, where $M_\sigma$ is the average of $\sigma$ over the generating set. This definition requires no hypothesis on the generating set beyond symmetry, but it is common in the literature on random walks to also assume that the generating set is closed under conjugation. In this case, Schur's lemma collapses each block to a scalar, and the ordering parameter is determined by the normalized character sum. Defining the ordering parameter as a mean of block eigenvalues retains a single number per irrep; although, it is worth noting that the ordering on $\widehat{G}$ is coarser than the ordering on the spectrum of the Laplacian itself.

A construction of this kind invites the objection that the Laplacian was an arbitrary starting point. To address this concern, we isolate four properties that any operator serving as a measure of smoothness should have, namely that it be Hermitian, that it commute with all left translations, that it annihilate the constant functions, and that the energy it assigns be unchanged under the complex conjugation of the argument. We then show that these properties alone do not narrow the ordering down; every real function on $\widehat{G}$ vanishing at the trivial representation and agreeing on conjugate pairs is realized. However, the ordering $\omega_K$ associated to inversion-closed unions $K=C\cup C^{-1}$ of conjugacy classes form a basis so that what the axioms provide is a coordinate system, in which an ordering is a weight vector recording which conjugacy classes count as small steps and how much they count. Two further requirements do cut the freedom down. Nonnegativity of that weight vector produces a positive semidefinite weighted Cayley-Laplacian, and declaring in addition that the elements of a fixed generating set are the incremental changes and that none counts for more than another pins the operator to the Cayley-Laplacian up to positive scale. In this sense, the Cayley-Laplacian is not one choice among many, and it is the axioms about the generating set rather than those about the operator that make it so.

We additionally show how to remove the assumption that the domain of the function is a group. For a function defined on a finite set that carries a transitive group action, the domain is identified with a coset space, and we lift the function to a function on the group. The theory developed for functions on groups then applies to the lifted functions. Here, a second layer of choice arises. The acting group determines which frequencies of the group are accessible since an irrep contributes to the function space precisely when it contains a vector fixed by the stabilizer. Meanwhile, the generating set orders whatever frequencies survive. These two choices are independent, and we exhibit pairs of groups acting on the same set for which the resulting notions of smoothness substantially differ. Finally, we observe that finiteness is used in only one essential place, namely in drawing the Cayley graph, and that the ordering parameter and its basic properties persist for any compact group.

The remainder of the paper is organized as follows. Section~\ref{sec:finite-groups} develops the construction for finite groups, establishes the block diagonalization of the Laplacian, defines the ordering parameter, and characterizes the irreps attaining its extreme values. Section~\ref{sec:characterization} contains the characterization of admissible operators and their induced orderings. Section~\ref{sec:group-action} extends the theory to sets carrying a group action and determines which frequencies survive the descent. Section~\ref{sec:examples} works through several examples, with figures illustrating how a single function changes its apparent smoothness when the generating set is changed. We use a final example to demonstrate how to interpret smoothness as an inductive bias and showcase how a designer might choose the various parameters of the smoothness model in Section~\ref{sec:prior}. Section~\ref{sec:compact} treats compact groups, and Section~\ref{sec:conclusion} gives concluding remarks.

\section{Smoothness on Finite Groups}\label{sec:finite-groups}

Given a finite group $G$, the space of functions $f:G\to\mathbb{C}$ is denoted $L^2(G)$, and it carries an inner product
\begin{equation}
    \langle f,h\rangle=\frac{1}{\lvert G\rvert}\sum_{g\in G}\overline{f(g)}h(g).
\end{equation}
This space contains the functions that we study here. Our goal is to introduce a notion of smoothness on $G$ which is analogous to that on $\mathbb{R}$. An orthonormal basis for this space is given by the normalized delta functions $\sqrt{\lvert G\rvert}\delta_g$ defined by $\delta_g(g')=\delta_{g,g'}$, where $\delta_{g,g'}$ denotes the Kronecker delta function which takes the value one if $g=g'$ and zero otherwise. In this basis, a function is expanded as
\begin{equation}
    f=\frac{1}{\sqrt{\lvert G\rvert}}\sum_{g\in G}f(g)\sqrt{\lvert G\rvert}\delta_g.
\end{equation}
By the Peter-Weyl theorem~\cite{fulton2013representation}, another orthonormal basis is given by the normalized matrix entries of the irreps of $G$, denoted $\sqrt{d_\sigma}\sigma_{ij}$ for irrep $\sigma$. The collection of all irreps is known as the dual $\widehat{G}$. When $G$ is abelian, all irreps are one-dimensional and equivalent to the characters of $G$, which are generally defined by $\chi_\sigma(g)=\operatorname{Tr}[\sigma(g)]$. When this is the case, the dual has the structure of a group~\cite{rudin1962fourier}, and although that structure is lost for non-abelian groups, the characters do still form a basis for the space of class functions. Note that every representation of a finite group is unitarizable, and so we make the assumption throughout this work that each irrep is unitary without loss of generality.

As noted in the introduction, smoothness on $\mathbb{R}$ is linked to the decay of the Fourier transform. In order to establish such a connection on $G$, we define the Fourier transform on $L^2(G)$ as a map $\mathcal{F}_G:L^2(G)\to\bigoplus_{\sigma\in\widehat{G}}\operatorname{End}(\mathcal{H}_\sigma)$ defined by
\begin{equation}
    \widehat{f}(\sigma):=(\mathcal{F}_Gf)(\sigma):=\frac{\sqrt{d_\sigma}}{\lvert G\rvert}\sum_{g\in G}f(g)\overline{\sigma(g)}.
\end{equation}
An analog of Parseval's identity holds and reads $\|f\|^2=\sum_{\sigma\in\widehat{G}}\|\widehat{f}(\sigma)\|_F^2$, where $\|\cdot\|_F$ denotes the Frobenius norm. Concentration of $\widehat{f}$ at low $\omega$ is therefore a statement about how the squared norm of $f$ distributes over $\widehat{G}$, and we shall have more to say about this in the next section. When $G$ is abelian, the codomain is equivalent to $L^2(\widehat{G})$, and the Fourier transform takes the form
\begin{equation}
    \widehat{f}(\chi)=\frac{1}{\lvert G\rvert}\sum_{g\in G}f(g)\overline{\chi(g)},
\end{equation}
where $\chi\in\widehat{G}$ is a character of $G$. In particular, when $G=\mathbb{Z}_d$, the characters are $\chi_k(x)=e^{2\pi i k x/d}$ with $k,x\in\mathbb{Z}_d$, and the Fourier transform takes the familiar form
\begin{equation}
    \widehat{f}(\chi_k)=\frac{1}{\lvert G\rvert}\sum_{x\in G}f(x)e^{-2\pi i k x/d}.
\end{equation}
When $G=\mathbb{Z}_2^n$, the characters are $\chi_k(x)=(-1)^{k\cdot x}$ with $k,x\in\mathbb{Z}_2^n$, and we recover \eqref{eq:z2n-fourier}.

As suggested by Belis et al.~\cite{belis2026spectral}, we interpret a smooth function to be a function with Fourier coefficients concentrated at low frequency. However, this assertion presupposes an ordering on the Fourier coefficients, and for non-abelian groups, a natural choice is not particularly obvious. The case of abelian groups is fortunately simpler, and we return to the case $G=\mathbb{Z}_2^n$ as an example. Observe that the character $(-1)^{k\cdot x}$ depends only on the $x$ coordinates for which the corresponding $k$ coordinate is nonzero; any change in $x_i$ for which $k_i=0$ does not affect the outcomes of the Fourier transform. Thus, the characters with more nonzero coordinates are more susceptible to a change in the function domain. This motivates the Hamming weight of $k$~\cite{odonnel2014analysis} as a natural choice for an ordering on $\widehat{G}$, so that smooth functions are those with Fourier coefficients concentrated in the low Hamming weight regime.

In order to define smoothness in an analogous way for a general finite group, let us now produce a method for equipping their Fourier coefficients with an ordering. To this end, let $G$ be a finite group and choose a generating set $S$. We assume that $S$ is symmetric, meaning it is closed under inversion: $s^{-1}\in S$ for all $s\in S$. The \emph{Cayley graph} associated to this choice of generating set is the graph with a vertex for every group element and edges between vertices whenever the group elements labeling them differ by a right multiplication by an element of $S$. The Cayley graph carries a natural ordering of the group elements in the form of a generalized Hamming weight given by the minimum number of generators needed to express a given group element. One recovers the Hamming weight by choosing $G=\mathbb{Z}_2^n$ and $S=\{e_i\}_{i=1}^n$, where $e_i=(0,\ldots,0,1,0,\ldots,0)$ denotes the standard basis vector with a one in the $i$-th component and zero otherwise. The Cayley graph treats elements of the generating set $S$ as the smallest possible perturbations of a group element, and so any notion of smoothness with respect to this structure should probe what happens to the value of a function on $G$ as its argument is perturbed by a generator. Thus, the notion of smoothness depends on the choice of generating set. We remark that this generalization appears in the study of symmetry-based quantum codes as a natural notion of distance~\cite{bradshaw2026symmetry}, and it would be interesting to determine whether the Hamming weight circuits of Rethinasamy et al.~\cite{rethinasamy2024logarithmic} can be extended to compute it.

With a notion of incremental change placed on the group, we now measure the change of $f$ under a small perturbation in its argument. We define an averaging operator $A:L^2(G)\to L^2(G)$ by
\begin{equation}
    (Af)(g)=\frac{1}{\lvert S\rvert}\sum_{s\in S}f(gs),
\end{equation}
which replaces the value of the function $f$ at the given $g$ with its average over all adjacent vertices in the Cayley graph. Then the Laplacian operator $\mathcal{L}=I-A$~\cite{babai1979spectra,chung1997spectral} defines the average change in $f$ over all adjacent vertices in the Cayley graph, equivalently all incremental changes with respect to the generalized Hamming weight. The reader familiar with the combinatorial Laplacian of a graph may wonder why the degree matrix does not appear in our definition. This is because we use the normalized Laplacian and the Cayley graph is $\lvert S\rvert$-regular, so that the normalized degree matrix becomes the identity operator. The following proposition makes the interpretation of $\mathcal{L}$ as a measure of change in $f$ precise.

\begin{proposition}\label{prop:quadratic}
    Let $f\in L^2(G)$. Then
    \begin{equation}
        \langle f,\mathcal{L}f\rangle =\frac{1}{2\lvert G\rvert\lvert S\rvert}\sum_{g\in G}\sum_{s\in S}\lvert f(g)-f(gs)\rvert^2.
    \end{equation}
\end{proposition}
\begin{proof}
    We first show that $A$ is Hermitian. Indeed,
    \begin{equation}
        \begin{aligned}
            \langle f,Ah\rangle &=\frac{1}{\lvert G\rvert}\sum_{g\in G}\overline{f(g)}\frac{1}{\lvert S\rvert}\sum_{s\in S}h(gs)\\
            &=\frac{1}{\lvert G\rvert}\frac{1}{\lvert S\rvert}\sum_{s\in S}\sum_{g\in G}\overline{f(g)}h(gs)\\
            &=\frac{1}{\lvert G\rvert}\frac{1}{\lvert S\rvert}\sum_{s\in S}\sum_{g\in G}\overline{f(gs^{-1})}h(g)\\
            &=\frac{1}{\lvert G\rvert}\sum_{g\in G}\left(\frac{1}{\lvert S\rvert}\sum_{s\in S}\overline{f(gs^{-1})}\right)h(g),
        \end{aligned}
    \end{equation}
    and since $S$ is symmetric, we may re-index the sum over $S$, producing
    \begin{equation}
        \langle f,Ah\rangle=\frac{1}{\lvert G\rvert}\sum_{g\in G}\left(\frac{1}{\lvert S\rvert}\sum_{s\in S}\overline{f(gs)}\right)h(g)=\langle Af,h\rangle.
    \end{equation}
    Thus, $A$ is Hermitian.

    We now prove the proposition. Expanding the right hand side, we have
    \begin{equation}
        \begin{aligned}
            \frac{1}{2\lvert G\rvert\lvert S\rvert}\sum_{g\in G}\sum_{s\in S}\lvert f(g)-f(gs)\rvert^2
            &=\frac{1}{2\lvert G\rvert\lvert S\rvert}\sum_{g\in G}\sum_{s\in S}\left(\lvert f(g)\rvert^2+\lvert f(gs)\rvert^2-2\operatorname{Re}(\overline{f(g)}f(gs))\right)\\
            &=\frac{1}{2}\left(\frac{1}{\lvert G\rvert}\sum_{g\in G}\lvert f(g)\rvert^2+\frac{1}{\lvert G\rvert}\sum_{g\in G}\lvert f(g)\rvert^2-2\operatorname{Re}\langle f,Af\rangle\right).
        \end{aligned}
    \end{equation}
    Now using the fact that $A$ is Hermitian, we have
    \begin{equation}
        \operatorname{Re}\langle f,Af\rangle=\langle f,Af\rangle,
    \end{equation}
    from which it follows that
    \begin{equation}
        \frac{1}{2\lvert G\rvert\lvert S\rvert}\sum_{g\in G}\sum_{s\in S}\lvert f(g)-f(gs)\rvert^2=\langle f,(I-A)f\rangle.
    \end{equation}
    Since $\mathcal{L}=I-A$, the proposition holds.
\end{proof}

We complete the construction of our ordering by block diagonalizing the Laplacian in the irrep basis. The eigenvectors of the Laplacian determine the direction of change, and the eigenvalues determine the magnitude of that change. The larger eigenvalues therefore correspond to larger amounts of change. However, we are looking for a scalar in each block; that is, one for each irrep. For this reason, we order the irreps by the average of the eigenvalues in each block.

\begin{proposition}\label{prop:block-diag}
    The Laplacian $\mathcal{L}$ is block diagonal with a block for every irreducible representation $\sigma$. Moreover, in each block, $\mathcal{L}$ acts as $I_{d_\sigma}\otimes(I_{d_\sigma}- M_\sigma)$, where $M_\sigma:=\frac{1}{\lvert S\rvert}\sum_{s\in S}\sigma(s)$.
\end{proposition}
\begin{proof}
    Let $\sigma$ be an irrep of $G$ and observe that
    \begin{equation}
            (A\sigma_{ij})(g)=\frac{1}{\lvert S\rvert}\sum_{s\in S}\sigma_{ij}(gs)
            =\frac{1}{\lvert S\rvert}\sum_{s\in S}\sum_{k=1}^{d_\sigma}\sigma_{ik}(g)\sigma_{kj}(s)
            =\sum_{k=1}^{d_\sigma}\sigma_{ik}(g)\frac{1}{\lvert S\rvert}\sum_{s\in S}\sigma_{kj}(s),
    \end{equation}
    where in the second equality, we have used the fact $\sigma(gs)=\sigma(g)\sigma(s)$ that $\sigma$ is a homomorphism. Defining $M_\sigma:=\frac{1}{\lvert S\rvert}\sum_{s\in S}\sigma(s)$, we have
    \begin{equation}
        (A\sigma_{ij})(g)=\sum_{k=1}^{d_\sigma}\sigma_{ik}(g)(M_\sigma)_{k,j}.
    \end{equation}
    That is, $A$ preserves $\sigma$ and $i$, but mixes along $j$. In other words, $A$ preserves the $\sigma$-isotypic Peter-Weyl block $\mathcal{A}_\sigma:=\operatorname{span}\{\sigma_{ij}\}$ and is therefore block diagonal with one block per irrep, and in each block, $A$ acts as $I_{d_\sigma}\otimes M_\sigma$. Noting that $\mathcal{L}=I-A$, we therefore have that $\mathcal{L}$ acts as $I_{d_\sigma}\otimes(I_{d_\sigma}- M_\sigma)$ in each block.
\end{proof}

By Proposition~\ref{prop:block-diag}, the eigenvalues of $\mathcal{L}$ in the block corresponding to irrep $\sigma$ are determined by the eigenvalues of $M_\sigma$. Since the Laplacian acts on the $\sigma$-isotypic Peter-Weyl block $\mathcal{A}_\sigma$ as $I_{d_\sigma}\otimes(I_{d_\sigma}-M_\sigma)$, it has up to $d_\sigma$ distinct eigenvalues, each of multiplicity $d_\sigma$. Their magnitudes dictate the amount of change in $f$ with an incremental change in its argument, and so we order the irreps by the average of the eigenvalues in each block.

\begin{definition}
    The constant $\omega(\sigma):=1-\frac{1}{d_\sigma}\operatorname{Tr}[M_\sigma]$ is called the ordering parameter of $\sigma$.
\end{definition}

If $\sigma'=U\sigma U^\dagger$ is an equivalent unitary irrep, then $M_{\sigma'}=UM_\sigma U^\dagger$, so $\operatorname{Tr}[M_{\sigma'}]=\operatorname{Tr}[M_\sigma]$ and $\omega(\sigma')=\omega(\sigma)$. The ordering parameter therefore depends only on the equivalence class and is well defined as a function of $\widehat{G}$. Since the eigenvalues of $M_\sigma$ are in $[-1,1]$ (we show $M_\sigma$ is Hermitian in Proposition~\ref{prop:admissible-block}), the ordering parameter is bounded by $0\le\omega(\sigma)\le2$. Indeed, since $\sigma$ is unitary,
\begin{equation}
    \| M_\sigma\|\le\frac{1}{\lvert S\rvert}\sum_{s\in S}\|\sigma(s)\|=1,
\end{equation}
and it follows that $\omega(\sigma)\in[0,2]$. In fact, the lower bound is only achieved by the trivial representation, and we can characterize when the upper bound is achieved by the structure of the Cayley graph.

\begin{proposition}
    Let $\sigma$ be an irrep of $G$ and let $S$ denote a generating set. Then $\omega(\sigma)=0$ if and only if $\sigma$ is the trivial representation, and $\max_\sigma\omega(\sigma)=2$ if and only if the Cayley graph associated to $(G,S)$ is bipartite, in which case the unique maximizer is the associated sign character.
\end{proposition}
\begin{proof}
    Let $\varepsilon\in\{\pm1\}$ and suppose $\omega(\sigma)=1-\varepsilon$. The mean of the eigenvalues of $M_\sigma$ is given by $\varepsilon$, which is an endpoint of the interval $[-1,1]$ containing all eigenvalues. This can only be the case if $M_\sigma=\varepsilon I_{d_\sigma}$. Let $v$ be a unit vector. Then by the unitarity of $\sigma$, we have
    \begin{equation}
        \frac{1}{\lvert S\rvert}\sum_{s\in S}\|(\sigma(s)-\varepsilon I_{d_\sigma})v\|^2=2(1-\varepsilon\langle v,M_\sigma v\rangle)=0.
    \end{equation}
    It follows that $\sigma(s)=\varepsilon I_{d_\sigma}$ for all $s\in S$. Hence, $\sigma(s_1\cdots s_k)=\varepsilon^k I_{d_\sigma}$, which means that every one-dimensional subspace is invariant, and so irreducibility forces $d_\sigma=1$.

    If $\varepsilon=1$, then $\sigma$ is trivial since $S$ generates $G$. If $\varepsilon=-1$, then $\sigma:G\to\{\pm1\}$ is a character with $\sigma\vert_S=-1$, so $H=\ker(\sigma)$ has index two and misses $S$, i.e. the graph is bipartite. Conversely, each case attains its bound. Indeed, the trivial character gives $M_\sigma=1$, and for $(G,S)$ with a bipartite Cayley graph, the sign character of the bipartition gives $M_\sigma=-1$. For the uniqueness, if $\sigma_1,\sigma_2$ are both $-1$ on $S$, then $\sigma_1\sigma_2$ is $1$ on $S$. But $S$ generates $G$, and so $\sigma_1\sigma_2$ is the trivial representation, establishing that $\sigma_2=\sigma_1^{-1}$. But $\sigma_1$ takes values in $\{\pm1\}$, so $\sigma_1^{-1}=\sigma_1$ and hence $\sigma_2=\sigma_1$.
\end{proof}

The ordering parameter is the average of the quadratic form of Proposition~\ref{prop:quadratic} over the block. Indeed, let $\{f_a\}_{a=1}^{d_\sigma^2}$ be an orthonormal basis of $\mathcal{A}_\sigma$. Since $\mathcal{L}$ acts on $\mathcal{A}_\sigma$ as $I_{d_\sigma}\otimes(I_{d_\sigma}-M_\sigma)$, its trace there is $d_\sigma\operatorname{Tr}[I_{d_\sigma}-M_\sigma]$, and dividing by $\dim(\mathcal{A}_\sigma)=d_{\sigma}^2$ gives
\begin{equation}
    \frac{1}{d_\sigma^2}\sum_{a=1}^{d_\sigma^2}\frac{1}{2\lvert G\rvert\lvert S\rvert}\sum_{g\in G}\sum_{s\in S}\lvert f_a(g)-f_a(gs)\rvert^2=\omega(\sigma).
\end{equation}
Equivalently, $\omega(\sigma)=\mathbb{E}\langle f,\mathcal{L}f\rangle$ for $f$ drawn uniformly from the unit sphere of $\mathcal{A}_\sigma$. For an individual unit $f\in\mathcal{A}_\sigma$, the energy $\langle f,\mathcal{L}f\rangle$ lies in $[\lambda_\textnormal{min}(I_{d_\sigma}-M_\sigma),\lambda_\textnormal{max}(I_{d_\sigma}-M_\sigma)]$, and the two agree with $\omega(\sigma)$ for every $f$ exactly when the block is scalar, which by Lemma~\ref{lem:conjugation-closed}, is the case when $S$ is closed under conjugation. 

\begin{lemma}\label{lem:conjugation-closed}
    If $S$ is closed under conjugation, then $I_{d_\sigma}\otimes M_\sigma=c_\sigma I_{d_\sigma^2}$ for some constant $c_\sigma$.
\end{lemma}
\begin{proof}
    Since $S$ is closed under conjugation, we have
    \begin{equation}
            \sigma(h)M_\sigma\sigma(h)^{-1}=\frac{1}{\lvert S\rvert}\sum_{s\in S}\sigma(hsh^{-1})
            =\frac{1}{\lvert S\rvert}\sum_{s\in S}\sigma(s)=M_\sigma.
    \end{equation}
    Then by Schur's lemma, $I_{d_\sigma}\otimes M_\sigma=c_\sigma I_{d_\sigma^2}$ for some constant $c_\sigma$.
\end{proof}

The reader may note that the choice of a mean to obtain a scalar from each block of the Laplacian seems arbitrary when one could just as well use the maximum eigenvalue or even the minimum. However, the normalized trace is the unique linear basis-independent scalar summary of a block normalized by $F(I)=1$; every linear functional satisfying $F(UBU^\dagger)=F(B)$ for all unitaries $U$ is a scalar multiple of the trace. As we will see, the linearity hypothesis is necessary for Corollary~\ref{cor:gen-set} and Theorem~\ref{thm:characterization} to hold. 

Importantly, computing $\omega$ never requires forming $\mathcal{L}$ as a $\lvert G\rvert\times\lvert G\rvert$ matrix. Since $\operatorname{Tr}[M_\sigma]=\frac{1}{\lvert S\rvert}\sum_{s\in S}\chi_\sigma(s)$, the ordering parameter at $\sigma$ is determined by a sum of $\lvert S\rvert$ character values. We stress that the character sum formula requires no hypothesis on $S$ beyond symmetry. What closure under conjugation adds is that the block is scalar, so that the mean is a lossless summary of the block rather than a lossy one. For $S$ not closed under conjugation, the number $\omega(\sigma)$ is still available, but it discards the spread of the block.

Given two choices $(G,S)$ and $(H,T)$ of groups and generating sets defining smooth structures, we can construct a new smooth structure on the product group. We do so in two distinct ways which differ in how we construct the generating set for the product group.

\begin{proposition}
    Let $(G,S)$ and $(H,T)$ be pairs of groups and generating sets with ordering parameters $\omega_G$ and $\omega_H$, respectively. Assume $e\notin S$ and $e\notin T$. Then the ordering parameter corresponding to the group $G\times H$ with generating set $(S\times\{e\})\cup(\{e\}\times T)$ is given by
    \begin{equation}
        \omega_{G\times H}(\sigma\otimes\rho)=\frac{\lvert S\rvert}{\lvert S\rvert+\lvert T\rvert}\omega_G(\sigma)+\frac{\lvert T\rvert}{\lvert S\rvert+\lvert T\rvert}\omega_H(\rho).
    \end{equation}
\end{proposition}
\begin{proof}
    The irreps of $G\times H$ are given by the tensor products of the irreps of $G$ and $H$, which have dimension $d_{\sigma\otimes\rho}=d_\sigma d_\rho$. Fixing an irrep $\sigma\otimes\rho$, we have
    \begin{equation}
    \begin{aligned}
        \omega_{G\times H}(\sigma\otimes\rho)
        &=1-\frac{1}{d_\sigma d_\rho}\frac{1}{\lvert S\rvert+\lvert T\rvert}\left(\sum_{s\in S}\chi_\sigma(s)d_\rho+\sum_{t\in T}d_\sigma\chi_\rho(t)\right)\\
        &=1-\frac{\lvert S\rvert}{\lvert S\rvert +\lvert T\rvert}(1-\omega_G(\sigma))-\frac{\lvert T\rvert}{\lvert S\rvert+\lvert T\rvert}(1-\omega_H(\rho))\\
        &=\frac{\lvert S\rvert}{\lvert S\rvert+\lvert T\rvert}\omega_G(\sigma)+\frac{\lvert T\rvert}{\lvert S\rvert+\lvert T\rvert}\omega_H(\rho).
    \end{aligned}
    \end{equation}
\end{proof}

\begin{proposition}
    Let $(G,S)$ and $(H,T)$ be pairs of groups and generating sets with ordering parameters $\omega_G$ and $\omega_H$, respectively. Then the ordering parameter corresponding to the group $G\times H$ with generating set $S\times T$ is given by
    \begin{equation}
        1-\omega_{G\times H}(\sigma\otimes\rho)=(1-\omega_G(\sigma))(1-\omega_H(\rho)).
    \end{equation}
\end{proposition}
\begin{proof}
    The irreps of $G\times H$ are given by the tensor products of the irreps of $G$ and $H$, which have dimension $d_{\sigma\otimes\rho}=d_\sigma d_\rho$. Fixing an irrep $\sigma\otimes\rho$, we have
    \begin{equation}
        \begin{aligned}
            1-\omega_{G\times H}(\sigma\otimes\rho)
            &=\frac{1}{d_\sigma d_\rho\lvert S\rvert\lvert T\rvert}\sum_{s\in S}\sum_{t\in T}\chi_\sigma(s)\chi_\rho(t)\\
            &=\left(\frac{1}{d_\sigma\lvert S\rvert}\sum_{s\in S}\chi_\sigma(s)\right)\left(\frac{1}{d_\rho\lvert T\rvert}\sum_{t\in T}\chi_\rho(t)\right)\\
            &=(1-\omega_G(\sigma))(1-\omega_H(\rho)).
        \end{aligned}
    \end{equation}
\end{proof}

The quantity $\omega(\sigma)$ is not new as a number. When $S$ is a union of conjugacy classes, the spectrum of $A$ is classical, computed by Diaconis and Shahshahani in the study of random walks on finite groups~\cite{diaconis1988group,diaconis1981generating}, and the block diagonalization of Proposition~\ref{prop:block-diag} is a standard result. Such generating sets are called quasi-abelian in the signal processing literature, and Proposition~\ref{prop:block-diag} without that hypothesis is the Laplacian form of the eigendecomposition of a weighted Cayley graph~\cite{rockmore2002fast,beck2024frames}. Our use of this quantity differs in that we retain a value at every irrep and read it as an ordering, whereas a mixing time argument typical of the random walk literature extracts a single eigenvalue and discards the remainder. Ordering frequencies by a Laplacian eigenvalue is likewise routine in graph signal processing, where a low pass filter on a graph is by definition the suppression of the large eigenvalues of $\mathcal{L}$~\cite{shuman2013emerging}, and it is the working definition of smoothness in landscape theory, where a correlated landscape is one whose amplitude spectrum sits on the low eigenvalues~\cite{stadler2002fitness}. The ordering parameter we define is a specialization of this concept to Cayley graphs, for which the corresponding Laplacian is block diagonal over the dual. The block at irrep $\sigma$ is determined by $M_\sigma$ alone, and the eigenvalue ordering therefore descends to the dual $\widehat{G}$, a set that exists prior to any diagonalization. As a consequence, the ordering parameter is a function $\omega:\widehat{G}\to\mathbb{R}$ that depends only on the pair $(G,S)$. For a generic graph, the Laplacian spectrum is a sorted list of real numbers with no domain attached, whereas Proposition~\ref{prop:block-diag} delivers the spectrum already partitioned by $\widehat{G}$, a set determined by $G$ alone. The ordering is therefore a function on a fixed domain, which allows the orderings induced by different generating sets to be compared pointwise.

Although the use of the Laplacian is natural due in part to its interpretation as a quantification of the average change in $f$ over all incremental steps in the Cayley graph, we have not yet established that this is the only such choice one can make to construct an ordering on $\widehat{G}$. In fact, there is some leeway in this choice, and we now determine exactly how much. 

\section{Characterization of the Ordering} \label{sec:characterization}
Call an operator $T:L^2(G)\to L^2(G)$ \emph{admissible} if it is Hermitian, commutes with all left translations, annihilates the constant functions, and satisfies $\langle\overline{f},T\overline{f}\rangle=\langle f,Tf\rangle$ for every $f\in L^2(G)$. These conditions are natural axioms for an operator replacing the Laplacian in our construction. Hermiticity makes each individual $\sigma$-block Hermitian so that its eigenvalues are real and their mean is a real number, left invariance says that no point of $G$ is distinguished, annihilation of the constant functions records that constant functions are considered smooth, and the last condition asks that the energy assigned to a function not depend on an overall choice of complex conjugate, which the quadratic form of Proposition~\ref{prop:quadratic} manifestly satisfies. Given an admissible $T$, we write $\omega_T(\sigma)$ for the mean of the eigenvalues of $T$ on the block $\mathcal{A}_\sigma$, extending the definition of $\omega(\sigma)$ from $\mathcal{L}$ to $T$.

Let $\mathcal{K}$ denote the collection of sets $C\cup C^{-1}$ as $C$ ranges over the nontrivial conjugacy classes of $G$. Each $K\in\mathcal{K}$ is symmetric, so $\mathcal{L}_K:=I-A_K$ is the Laplacian of the Cayley graph on $G$ generated by $K$ whenever $K$ generates, and is the Laplacian of a disconnected Cayley graph otherwise. We write $\omega_K$ for the associated ordering parameter. Finally, for $\varphi\in L^2(G)$, let
\begin{equation}
    (T_\varphi f)(g)=\sum_{h\in G}\varphi(g^{-1}h)f(h)
\end{equation}
so that $\mathcal{L}=T_{\delta_e-\mu_S}$, where $\mu_S$ is the uniform probability measure on $S$. We begin by identifying the admissible operators concretely, which amounts to the standard observation that left invariance forces convolution~\cite{terras1999fourier}.

\begin{lemma}\label{lem:admissible-form}
    An operator $T$ on $L^2(G)$ is admissible if and only if $T=T_\varphi$ for a real valued $\varphi$ satisfying $\varphi(x^{-1})=\varphi(x)$ and $\sum_{x\in G}\varphi(x)=0$.
\end{lemma}
\begin{proof}
    Let $T$ have matrix $\kappa$ so that $(Tf)(g)=\sum_{h\in G}\kappa(g,h)f(h)$, and write $(L_af)(g)=f(a^{-1}g)$. Then
    \begin{equation}\label{eq:left}
        (L_aTf)(g)=\sum_{h\in G}\kappa(a^{-1}g,h)f(h)
    \end{equation}
    and
    \begin{equation}\label{eq:right}
        (TL_af)(g)=\sum_{h\in G}\kappa(g,ah)f(h).
    \end{equation}
    Comparing \eqref{eq:left} with \eqref{eq:right} shows that $T$ commutes with every left translation if and only if $\kappa(a^{-1}g,h)=\kappa(g,ah)$ for all $a,g,h\in G$. Taking $a=g$ gives $\kappa(g,gh)=\kappa(e,h)$, so that $\kappa(g,h)=\varphi(g^{-1}h)$ with $\varphi(x):=\kappa(e,x)$, and hence $T=T_\varphi$. The converse is the same computation read backwards.

    A direct computation gives $T_\varphi^\dagger=T_{\varphi^*}$, where $\varphi^*(x)=\overline{\varphi(x^{-1})}$, so $T_\varphi$ is Hermitian if and only if $\varphi(x)=\overline{\varphi(x^{-1})}$. Writing $\overline{T}$ for the operator with matrix $\overline{\kappa}$, we have
    \begin{equation}
        \overline{\langle\overline{f},T\overline{f}\rangle}=\langle f,\overline{T}f\rangle,
    \end{equation}
    and Hermiticity makes both quadratic forms real. Thus, the fourth axiom of admissibility holds if and only if $\overline{T}=T$; that is, if and only if $\varphi$ is real valued. Combining the two conditions, $\varphi$ is real and satisfies $\varphi(x^{-1})=\varphi(x)$. Finally, $T_\varphi1=(\sum_x\varphi(x))1$, so $T$ annihilates the constant functions if and only if $\sum_x\varphi(x)=0$.
\end{proof}

Each $\mathcal{L}_K$ is of this form since $\delta_e-\mu_K$ is real, is invariant under inversion due to symmetry of $K$, and has total mass zero. The same is true of $\mathcal{L}$ itself. Moreover, the block structure of Proposition~\ref{prop:block-diag} persists for every admissible operator.

\begin{proposition}\label{prop:admissible-block}
    Let $T=T_\varphi$ be admissible. Then $T$ acts on $\mathcal{A}_\sigma$ as $I_{d_\sigma}\otimes M_\sigma^\varphi$, where $M_\sigma^\varphi:=\sum_{x\in G}\varphi(x)\sigma(x)$ is Hermitian, and
    \begin{equation}\label{eq:omegaT}
        \omega_T(\sigma)=\frac{1}{d_\sigma}\operatorname{Tr}[M_\sigma^\varphi]=\frac{1}{d_\sigma}\sum_{x\in G}\varphi(x)\chi_\sigma(x).
    \end{equation}
    Moreover, we have
    \begin{equation}
        \omega_{T_\varphi}=\omega_{T_{\varphi^\natural}},
    \end{equation}
    where $\varphi^\natural(x)=\frac{1}{\lvert G\rvert}\sum_{a\in G}\varphi(axa^{-1})$.
\end{proposition}
\begin{proof}
    As in the proof of Proposition~\ref{prop:block-diag}, we have
    \begin{equation}
        (T_\varphi\sigma_{ij})(g)=\sum_{x\in G}\varphi(x)\sigma_{ij}(gx)=\sum_{k=1}^{d_\sigma}\sigma_{ik}(g)(M_\sigma^\varphi)_{kj},
    \end{equation}
    so $T_\varphi$ preserves $\mathcal{A}_\sigma$ and acts there as $I_{d_\sigma}\otimes M_\sigma^\varphi$. Hermiticity of the block follows from Lemma~\ref{lem:admissible-form} and unitarity of $\sigma$ since
    \begin{equation}
        (M_\sigma^\varphi)^\dagger=\sum_{x\in G}\varphi(x)\sigma(x)^{-1}=\sum_{y\in G}\varphi(y^{-1})\sigma(y)=M_\sigma^\varphi.
    \end{equation}
    Taking the mean of the $d_\sigma$ eigenvalues of the block gives \eqref{eq:omegaT}. The final claim follows by averaging the substitution $x\mapsto axa^{-1}$ over $a\in G$ in \eqref{eq:omegaT}, using that $\chi_\sigma$ is a class function.
\end{proof}

Taking $\varphi=\delta_e-\mu_S$ in \eqref{eq:omegaT} recovers $\omega(\sigma)=1-\operatorname{Tr}[M_\sigma]/d_\sigma$, so the ordering parameter of Proposition~\ref{prop:block-diag} is the case $T=\mathcal{L}$. Proposition~\ref{prop:admissible-block} also further justifies the choice of the mean in the definition of the ordering parameter because the operator $T_{\varphi^\natural}$ is diagonal in the irrep basis when $\varphi=\delta_e-\mu_S$, as we show in the next corollary. Thus, an ordering induced by a non-diagonal Cayley-Laplacian is not just a lossy description of its spectrum; rather, it takes into account the exact spectrum of the conjugation-symmetrized weighted Cayley-Laplacian.

\begin{corollary}\label{cor:diagonal-laplacian}
    Let $\varphi=\delta_e-\mu_S$ so that $T_\varphi=\mathcal{L}$. Then $\mathcal{L}^\natural:=T_{\varphi^\natural}$ acts on $\mathcal{A}_\sigma$ as $\omega(\sigma) I_{d_\sigma^2}$.
\end{corollary}
\begin{proof}
    By Proposition~\ref{prop:admissible-block}, $T_{\varphi^\natural}$ acts on $\mathcal{A}_\sigma$ as $I_{d_\sigma}\otimes M_\sigma$ with
    \begin{equation}
        M_\sigma=\sum_{x\in G}\varphi^\natural(x)\sigma(x).
    \end{equation}
    Then for all $g\in G$ we have
    \begin{equation}
        \sigma(g)M_\sigma\sigma(g^{-1})=\sum_{x\in G}\varphi^\natural(x)\sigma(gxg^{-1})=\sum_{x\in G}\varphi^\natural(g^{-1}xg)\sigma(x).
    \end{equation}
    But $\varphi^\natural$ is a class function, and so $\sigma(g)M_\sigma=M_\sigma\sigma(g)$ for all $g\in G$. Thus, by Schur's lemma, $M_\sigma=c_\sigma I$ for some constant $c_\sigma$. Since $\omega_{T_\varphi}=\omega_{T_\varphi^\natural}$, we have $c_\sigma=\omega(\sigma)$.
\end{proof}

Combining Proposition~\ref{prop:admissible-block} and Corollary~\ref{cor:diagonal-laplacian} shows that every ordering parameter comes from a diagonal admissible operator. We now show that no admissible operator produces an ordering outside the real span of those coming from the Laplacians $\mathcal{L}_K$.

\begin{theorem}\label{thm:characterization}
    The Laplacians $\{\mathcal{L}_K\}_{K\in\mathcal{K}}$ are linearly independent, and for every admissible $T$, there is an $\alpha\in\mathbb{R}^{\lvert\mathcal{K}\rvert}$ with
    \begin{equation}\label{eq:admissible-scalar}
        \omega_T=\sum_{K\in\mathcal{K}}\alpha_K\omega_K.
    \end{equation}
\end{theorem}
\begin{proof}
    Let $T=T_\varphi$ be admissible. By Lemma~\ref{lem:admissible-form}, the function $\varphi^\natural$ is again real, inversion invariant, and of total mass zero, and it is in addition a class function, hence constant on $\{e\}$ and on each member of $\mathcal{K}$. Since these sets partition $G$,
    \begin{equation}
        \varphi^\natural=\lambda_e\delta_e+\sum_{K\in\mathcal{K}}\lambda_K\mu_K
    \end{equation}
    with $\lambda_e,\lambda_K\in\mathbb{R}$, and vanishing total mass forces $\lambda_e+\sum_{K\in\mathcal{K}}\lambda_K=0$. Substituting $A_K=I-\mathcal{L}_K$ then gives
    \begin{equation}
        T_{\varphi^\natural}=\left(\lambda_e+\sum_{K\in\mathcal{K}}\lambda_K\right)I-\sum_{K\in\mathcal{K}}\lambda_K\mathcal{L}_K=-\sum_{K\in\mathcal{K}}\lambda_K\mathcal{L}_K,
    \end{equation}
    and \eqref{eq:admissible-scalar} follows from Proposition~\ref{prop:admissible-block} with $\alpha_K=-\lambda_K$.

    Suppose now that $\sum_K\alpha_K\mathcal{L}_K=0$ for some real $\alpha$. Then $\sum_K\alpha_KA_K=(\sum_K\alpha_K)I$, and comparing the associated functions gives $\sum_K\alpha_K\mu_K=(\sum_K\alpha_K)\delta_e$. Evaluating at any $x\neq e$, which lies in exactly one member $K$ of $\mathcal{K}$, gives $\alpha_K/\lvert K\rvert=0$. Hence, every $\alpha_K$ vanishes, and the $\mathcal{L}_K$ are linearly independent.
\end{proof}

The theorem does not select a unique ordering. Rather, it identifies the $\omega_K$ as a canonical basis for the space of orderings arising from admissible operators. The next result identifies exactly which orderings arise and shows that this space is unconstrained in any way that a candidate ordering would not already satisfy. Reality is forced by any construction taking eigenvalue means of a Hermitian operator, agreement on conjugate pairs by any construction not distinguishing a representation from its dual, and vanishing at the trivial representation is a normalization. Admissibility therefore supplies coordinates rather than a restriction, in which an ordering becomes a weight vector on $\mathcal{K}$ recording which conjugacy classes count as small steps and how much they count.

\begin{theorem}\label{thm:image}
    The map $T\mapsto\omega_T$ carries the admissible operators onto the real valued functions on $\widehat{G}$ which vanish at the trivial representation and agree on complex conjugate pairs. Its kernel is $\{T_\varphi:\varphi^\natural=0\}$, and it restricts to a linear bijection on $\operatorname{span}_\mathbb{R}\{\mathcal{L}_K\}$, which is exactly the set of admissible operators commuting with all right translations.
\end{theorem}
\begin{proof}
    By Lemma~\ref{lem:admissible-form}, the assignment $\varphi\mapsto T_\varphi$ is a bijection, so it suffices to study the map $\varphi\mapsto\omega_{T_\varphi}$. By Proposition~\ref{prop:admissible-block}, this map factors as $\varphi\to\varphi^\natural\mapsto\omega_{T_{\varphi^\natural}}$, and by \eqref{eq:omegaT}, its second stage sends a class function $\psi$ to the function $\sigma\mapsto\frac{1}{d_\sigma}\sum_{x\in G}\psi(x)\chi_\sigma(x)$ on $\widehat{G}$; that is, to the coefficients of $\psi$ in the character basis divided by $d_\sigma$. Since the irreducible characters form a basis for the class functions, that second stage is a linear bijection onto all functions from $\widehat{G}$ to $\mathbb{C}$. Writing $\varphi^\vee(x)=\varphi(x^{-1})$ and using $\chi_{\overline{\sigma}}=\overline{\chi_\sigma}$ together with $\chi_\sigma(x^{-1})=\overline{\chi_\sigma(x)}$, we have
    \begin{equation}
        \omega_{T_{\overline{\varphi}}}(\sigma)=\overline{\omega_{T_\varphi}(\overline{\sigma})}
    \end{equation}
    and
    \begin{equation}
        \omega_{T_{\varphi^\vee}}(\sigma)=\omega_{T_\varphi}(\overline{\sigma}),
    \end{equation}
    and $\omega_{T_\varphi}(\sigma_\textnormal{triv})=\sum_{x\in G}\varphi(x)$. Hence, $\varphi^\natural$ is real if and only if $\omega_T(\sigma)=\overline{\omega_T(\overline{\sigma})}$, it is inversion invariant if and only if $\omega_T(\overline{\sigma})=\omega_T(\sigma)$, and it has vanishing total mass if and only if $\omega_T(\sigma_\textnormal{triv})=0$. The first two conditions together say exactly that $\omega_T$ is real valued and agrees on complex conjugate pairs, and the kernel is as stated.

    For the last claim, write $(R_bf)(g)=f(gb)$ and substitute $h\mapsto hb^{-1}$ to obtain
    \begin{equation}
        (T_\varphi R_bf)(g)=\sum_{h\in G}\varphi(g^{-1}hb^{-1})f(h)
    \end{equation}
    and
    \begin{equation}
        (R_bT_\varphi f)(g)=\sum_{h\in G}\varphi(b^{-1}g^{-1}h)f(h),
    \end{equation}
    so $T_\varphi$ commutes with every right translation if and only if $\varphi(xy)=\varphi(yx)$ for all $x,y\in G$; that is, if and only if $\varphi=\varphi^\natural$. On that subspace, the map is therefore injective, and by the proof of Theorem~\ref{thm:characterization}, the subspace is $\operatorname{span}_\mathbb{R}\{\mathcal{L}_K\}$.
\end{proof}

\begin{proposition}\label{prop:positivity}
    If $\alpha_K\ge0$ for every $K\in\mathcal{K}$, then $T_{\varphi^\natural}$ is positive semidefinite and is the Laplacian of the weighted Cayley graph on $G$ carrying weight $\alpha_K/\lvert K\rvert$ on the edges labeled by $K$.
\end{proposition}
\begin{proof}
    Applying Proposition~\ref{prop:quadratic} termwise to $T_{\varphi^\natural}=\sum_K\alpha_K\mathcal{L}_K$ gives
    \begin{equation}
        \langle f,T_{\varphi^\natural}f\rangle=\sum_{K\in\mathcal{K}}\frac{\alpha_K}{2\lvert G\rvert\lvert K\rvert}\sum_{g\in G}\sum_{s\in K}\lvert f(g)-f(gs)\rvert^2\ge0,
    \end{equation}
    which is the Dirichlet energy of the weighted Cayley graph described.
\end{proof}

Admissibility alone does not imply positivity. Proposition~\ref{prop:positivity} identifies a natural Dirichlet subclass: non-negative class coefficients produce positive-semidefinite weighted Cayley-Laplacians. Specializing to $T=\mathcal{L}$ makes the weights explicit and exhibits the ordering induced by an arbitrary symmetric generating set as a convex combination of those induced by conjugation closed ones.

\begin{corollary}\label{cor:gen-set}
   Let $S=S^{-1}$ be a generating set with $e\notin S$. Then
   \begin{equation}
       \omega=\sum_{K\in\mathcal{K}}\frac{\lvert S\cap K\rvert}{\lvert S\rvert}\omega_K,
   \end{equation}
   so the weights $\alpha_K=\lvert S\cap K\rvert/\lvert S\rvert$ form a probability vector on $\mathcal{K}$.
\end{corollary}
\begin{proof}
    Symmetry of $S$ gives $\lvert S\cap C\rvert=\lvert S\cap C^{-1}\rvert$ for every conjugacy class $C$, so averaging $\mu_S$ over conjugation yields $\mu_S^\natural=\sum_K\alpha_K\mu_K$ with $\sum_K\alpha_K=1$. Hence, $\varphi^\natural=\sum_K\alpha_K(\delta_e-\mu_K)$, and the claim follows from Proposition~\ref{prop:admissible-block}.
\end{proof}

The positivity hypothesis of Proposition~\ref{prop:positivity} has a probabilistic reading. Writing $T_\varphi f(g)=\sum_{x\ne e}\varphi(x)[f(gx)-f(g)]$, which is the content of vanishing total mass, the operator $-T_{\varphi^\natural}$ is the generator of a symmetric convolution Markov semigroup precisely when $\varphi^\natural\le0$ away from the identity, and by the computation above, this holds exactly when every $\alpha_K$ is nonnegative.

The number of distinct $\omega_K$ is quantified by the size of $\mathcal{K}$, which can be determined to a great extent. Indeed, inversion acts on the set of conjugacy classes as an involution, and $\mathcal{K}$ is the set of its orbits on the nontrivial classes. So, if $G$ has $k$ conjugacy classes and $r$ of them are real (meaning $C=C^{-1}$, so that the character is real on $C$)~\cite{isaacs1994character}, the non-real ones pair off and we have
\begin{equation}
    \lvert\mathcal{K}\rvert=r-1+\frac{k-r}{2}=\frac{k+r}{2}-1,
\end{equation}
where the $-1$ is included to remove the class of the identity, which is always real. When $G$ is ambivalent, meaning every element is conjugate to its inverse, $r=k$ and $\lvert\mathcal{K}\rvert=k-1$, so $\mathcal{K}$ is just the nontrivial conjugacy classes. This is the case for all symmetric and dihedral groups, for example. However, it does not hold in general, as exhibited by the group $\mathbb{Z}_3$, which has $k=3$, $r=1$, and $\lvert\mathcal{K}\rvert=1$.

There are two more assumptions one can make beyond admissibility which narrow down the choice of ordering completely. We separate them from the above because they deal not with the group but with the generating set for the group. If we interpret the choice of $S$ as a declaration of which group elements count as incremental changes, then the following two assumptions are natural. First, we can assume that an admissible operator $T_\varphi$ is $S$-local in the sense that the support of $\varphi$ is contained in $\{e\}\cup S$. The second is that $T_\varphi$ is $S$-uniform in the sense that $\varphi$ is constant on $S$. Locality is the assumption that elements of $S$ count as incremental changes, and uniformity is the assumption that no element of $S$ counts for more than another. With these additional assumptions, the induced Laplacian is narrowed down to the Laplacian of Section~\ref{sec:finite-groups} up to a rescaling, which does not affect the ordering of irreps provided the constant is strictly positive; a negative constant reverses the ordering and a zero constant destroys it. Requiring positive semidefiniteness as in Proposition~\ref{prop:positivity} rules out the negative case, and requiring that $T$ be nonzero rules out the remaining one.

\begin{proposition}\label{prop:local-uniform}
    Let $S=S^{-1}$ generate $G$ with $e\notin S$. An admissible $T$ is $S$-local and $S$-uniform if and only if $T=c\mathcal{L}$ for some constant $c\in\mathbb{R}$.
\end{proposition}
\begin{proof}
    By Lemma~\ref{lem:admissible-form}, $T=T_\varphi$ for some real-valued $\varphi$ satisfying $\varphi(x^{-1})=\varphi(x)$ and $\sum_{x\in G}\varphi(x)=0$. If $T$ is $S$-local and $S$-uniform, then $\varphi=\lambda_e\delta_e+\lambda\delta_S$, where $\delta_S$ is the indicator function on $S$ and $\lambda_e,\lambda\in\mathbb{R}$. It follows that
    \begin{equation}
    \begin{aligned}
        (T_\varphi f)(g)=\sum_{h\in G}(\lambda_e\delta_e+\lambda\delta_S)(g^{-1}h)f(h)
        =\lambda_ef(g)+\lambda\sum_{s\in S}f(gs)
        =((\lambda_eI+\lambda\lvert S\rvert A)f)(g).
    \end{aligned}        
    \end{equation}
    Since $\varphi$ has vanishing total mass, $\lambda_e=-\lambda\lvert S\rvert$, and so $T=c\mathcal{L}$ for a constant $c\in\mathbb{R}$. Conversely, if $T=c\mathcal{L}$, then $\varphi=c(\delta_e-\mu_S)$, which is $S$-local and $S$-uniform.
\end{proof}

We stress that Theorem~\ref{thm:characterization} characterizes the orderings available once one fixes the group. If one additionally declares which elements of $G$ count as incremental changes, the operator is pinned down up to a scaling factor. We therefore view our choice of the Laplacian in Section~\ref{sec:finite-groups} not as an arbitrary choice, but as the only choice satisfying the several natural axioms outlined here.

In the first paragraph of the introduction, we noted that smoothness on $\mathbb{R}$ can be viewed as a constraint on the Fourier transform of the function of interest. With our ordering parameter in hand, we may now produce an analog of this statement for the case of finite groups. Specifically, we show that a bandlimiting hypothesis on the Fourier coefficients implies an upper bound on the forward difference of a function and vice versa.

\begin{theorem}
    Let $f\in L^2(G)$. If $\widehat{f}$ is supported on $\{\sigma:\omega(\sigma)\le\lambda\}$, then 
    \begin{equation}
        \langle f,\mathcal{L}^\natural f\rangle\le\lambda\|f\|^2.
    \end{equation}
    Conversely, if $\lvert f(g)-f(gx)\rvert\le\Lambda$ for all $g$ and all $x$ in each $K$ intersecting $S$, then
    \begin{equation}
        \sum_{\omega(\sigma)>\lambda}\|\widehat{f}(\sigma)\|^2_F<\frac{\Lambda^2}{2\lambda}.
    \end{equation}
\end{theorem}
\begin{proof}
    Let $f_\sigma$ denote the projection of $f$ onto the $\sigma$-block. By Corollary~\ref{cor:diagonal-laplacian}, $\mathcal{L}^\natural$ acts on the $\sigma$ block as multiplication by $\omega(\sigma)$. Thus, $\mathcal{L}^\natural f_\sigma=\omega(\sigma)f_\sigma$. The blocks are mutually orthogonal, so we have
    \begin{equation}
        \langle f,\mathcal{L}^\natural f\rangle=\left\langle\sum_{\sigma\in\widehat{G}} f_\sigma,\sum_{\tau\in\widehat{G}} \omega(\tau)f_\tau\right\rangle=\sum_{\sigma\in\widehat{G}}\omega(\sigma)\|f_\sigma\|^2.
    \end{equation}
    By Parseval's identity, it follows that
    \begin{equation}\label{eq:natural-quadratic}
        \langle f,\mathcal{L}^\natural f\rangle=\sum_{\sigma\in\widehat{G}}\omega(\sigma)\|\widehat{f}_\sigma\|_F^2.
    \end{equation}
    If $\widehat{f}$ is supported on $\{\sigma:\omega(\sigma)\le\lambda\}$, it therefore follows that
    \begin{equation}
        \langle f,\mathcal{L}^\natural f\rangle=\sum_{\omega(\sigma)\le\lambda}\omega(\sigma)\|\widehat{f}_\sigma\|^2_F\le\lambda\sum_{\omega(\sigma)\le\lambda}\|\widehat{f}_\sigma\|_F^2,
    \end{equation}
    and by Parseval's identity, the right hand side is $\lambda\|f\|^2$.

    For the converse, suppose $\lvert f(x)-f(gx)\rvert\le\Lambda$ for all $g$ and all $x$ in each $K$ intersecting $S$. Then by Theorem~\ref{thm:characterization}, there are constants $\alpha_K$ such that $\mathcal{L}^\natural=\sum_K\alpha_K\mathcal{L}_K$ and by Proposition~\ref{prop:quadratic}, we have
    \begin{equation}
        \langle f,\mathcal{L}^\natural f\rangle=\sum_{K\in\mathcal{K}}\alpha_K\langle f,\mathcal{L}_Kf\rangle\le\frac{\Lambda^2}{2}\sum_{K\in\mathcal{K}}\alpha_K=\frac{\Lambda^2}{2}.
    \end{equation}
    Applying this inequality to \eqref{eq:natural-quadratic} produces
    \begin{equation}
        \frac{\Lambda^2}{2}\ge\langle f,\mathcal{L}^\natural f\rangle
        =\sum_{\sigma\in\widehat{G}}\omega(\sigma)\|\widehat{f}(\sigma)\|^2_F
        \ge\sum_{\omega(\sigma)>\lambda}\omega(\sigma)\|\widehat{f}(\sigma)\|^2_F
        \ge\lambda\sum_{\omega(\sigma)>\lambda}\|\widehat{f}(\sigma)\|^2_F,
    \end{equation}
    which is what we set out to prove.
\end{proof}

This theorem justifies the interpretation of smoothness as having Fourier support concentrated at low frequency. Indeed, an upper bound on the ordering parameter, restricting the Fourier support, produces a bound on the quadratic form associated to the conjugation-symmetrized Laplacian, which measures the amount that a function value changes when its domain is changed incrementally. Conversely, if one upper bounds the change in the function, this produces an upper bound on the weight of the Fourier support in the upper end of the spectrum.

\section{Smoothness on Sets Carrying a Group Action}\label{sec:group-action}
The theory of the previous sections applies to functions on a group, but it can be extended to functions on sets carrying a group action. This section removes the restriction using standard methods in harmonic analysis and is briefly discussed by Belis et al.~\cite{belis2026spectral}. We let a group act on the domain of the function and lift the function to a function on the group, where the theory we developed applies. This adds another layer to the construction; the ordering, and hence the notion of smoothness, is now a consequence of which group one chooses to act with. In Section~\ref{sec:examples}, we will discuss how two different groups acting on the same function domain can induce entirely different simplicity biases.

Let $X$ be a finite set and suppose a finite group $G$ acts transitively on $X$, so that for any $x,y\in X$, there is a $g\in G$ with $g\cdot x=y$. Fix a basepoint $x_0\in X$ and let
\begin{equation}
    H=\{h\in G:h\cdot x_0=x_0\}
\end{equation}
be its stabilizer. The map $Hg\mapsto g^{-1}\cdot x_0$ is then a bijection, giving the identification $X\cong H\setminus G$ and in particular, $\lvert X\rvert=\lvert G\rvert/\lvert H\rvert$. Transitivity is not much of a restriction in practice, since a group acting with several orbits can be treated one orbit at a time.

Given a function $f:X\to\mathbb{C}$, define its \emph{lift} $f^\uparrow: G\to\mathbb{C}$ by
\begin{equation}
    f^{\uparrow}(g)=f(g^{-1}\cdot x_0).
\end{equation}
For any $h\in H$, we have $f^\uparrow(hg)=f(g^{-1}\cdot(h^{-1}\cdot x_0))=f^\uparrow(g)$, so the lift is constant on the left cosets $Hg$. Write
\begin{equation}
    L^2(G)^H=\{F\in L^2(G):F(hg)=F(g)\ \forall h\in H\}
\end{equation}
for the space of such functions. Lifting is a linear isomorphism $L^2(X)\to L^2(G)^H$, and both spaces have dimension $\lvert X\rvert$.

For the theory above to descend, the averaging operator must preserve $L^2(G)^H$ so that it can be thought of as a map $A:L^2(G)^H\to L^2(G)^H$. It does, as we now show.
\begin{lemma}\label{lem:descend}
    Let $H\le G$ be any subgroup. Then $A$ preserves $L^2(G)^H$, and hence so does $\mathcal{L}=I-A$.
\end{lemma}
\begin{proof}
    Let $F\in L^2(G)^H$ and $h\in H$. Then
    \begin{equation}
        (AF)(hg)=\frac{1}{\lvert S\rvert}\sum_{s\in S}F(hgs)=\frac{1}{\lvert S\rvert}\sum_{s\in S}F(gs)=(AF)(g),
    \end{equation}
    where the second equality applies the left $H$-invariance of $F$ to each argument $gs$ separately.
\end{proof}

Concretely, the operator induced on $X$ is the averaging operator of the Schreier graph. Define $A_X:L^2(X)\to L^2(X)$ by
\begin{equation}
    (A_Xf)(x)=\frac{1}{\lvert S\rvert}\sum_{s\in S}f(s\cdot x).
\end{equation}
Substituting $x=g^{-1}\cdot x_0$ into the definition of the lift gives
\begin{equation}
    (Af^\uparrow)(g)=\frac{1}{\lvert S\rvert}\sum_{s\in S}f(s^{-1}g^{-1}\cdot x_0)=(A_Xf)(g^{-1}\cdot x_0),
\end{equation}
where the last equality uses symmetry of $S$. Hence, $(A_Xf)^\uparrow=Af^\uparrow$, and lifting intertwines $\mathcal{L}$ with the Laplacian of the Schreier graph of $(X,S)$. In particular, $\omega$ may be computed on $X$ directly, without passing through $G$.

Not every Fourier coefficient of $G$ is visible on $X$. To see which are, introduce the averaging operator over the stabilizer,
\begin{equation}
    (\Pi F)(g)=\frac{1}{\lvert H\rvert}\sum_{h\in H}F(hg).
\end{equation}
Since $H$ is a group, $\Pi$ is idempotent and Hermitian, and its image is exactly $L^2(G)^H$. It is therefore the orthogonal projection onto the space of lifted functions. The question of which frequencies survive is the question of how $\Pi$ interacts with the blocks of Proposition~\ref{prop:block-diag}, and the computation is the one already performed there.

\begin{proposition}\label{prop:restricted-fourier}
    The projection $\Pi$ preserves the block $\mathcal{A}_\sigma=\operatorname{span}\{\sigma_{ij}\}$, and
    \begin{equation}
        \dim(\mathcal{A}_\sigma\cap L^2(G)^H)=d_\sigma m_\sigma,
    \end{equation}
    where $m_\sigma=\frac{1}{\lvert H\rvert}\sum_{h\in H}\chi_\sigma(h)$. Moreover, $\sum_\sigma d_\sigma m_\sigma=\lvert X\rvert$.
\end{proposition}
\begin{proof}
    Let $P_\sigma^H=\frac{1}{\lvert H\rvert}\sum_{h\in H}\sigma(h)$. Similar to the proof of Proposition~\ref{prop:block-diag}, using $\sigma(hg)=\sigma(h)\sigma(g)$, we have
    \begin{equation}
        (\Pi\sigma_{ij})(g)=\frac{1}{\lvert H\rvert}\sum_{h\in H}\sigma_{ij}(hg)=\sum_{k=1}^{d_\sigma}\left(P_\sigma^H\right)_{ik}\sigma_{kj}(g),
    \end{equation}
    so $\Pi$ preserves $\sigma$ and the index $j$ and acts on the index $i$ by the matrix $P_\sigma^H$. In particular, $\mathcal{A}_\sigma$ is preserved. Re-indexing the double sum by $h\mapsto hh'$ shows $(P_\sigma^H)^2=P_\sigma^H$, so $P_\sigma^H$ is idempotent and its rank equals its trace,
    \begin{equation}
        \operatorname{rank}(P_\sigma^H)=\operatorname{Tr}\left[P_\sigma^H\right]=\frac{1}{\lvert H\rvert}\sum_{h\in H}\chi_\sigma(h)=m_\sigma.
    \end{equation}
    The $i$ index is confined to an $m_\sigma$-dimensional image while the $j$ index remains free, giving $d_\sigma m_\sigma$ surviving dimensions in the block. Summing over $\sigma$ and using $\dim(L^2(G)^H)=\lvert G\rvert/\lvert H\rvert=\lvert X\rvert$ gives the last claim.
\end{proof}

Proposition~\ref{prop:restricted-fourier} says that $\sigma$ contributes to $L^2(X)$ exactly when $m_\sigma>0$; that is, exactly when $\mathcal{H}_\sigma$ contains a nonzero $H$-fixed vector, and that it then contributes $d_\sigma m_\sigma$ dimensions rather than the $d_\sigma^2$ it occupies in $L^2(G)$. The frequencies visible on $X$ are therefore a subset of $\widehat{G}$, and the ordering induced on $X$ is the restriction of $\omega$ to that subset. Two extreme cases are worth recording. The trivial representation always survives since $m_\sigma=1$, so the constant functions are present on $X$ and remain the smoothest, as they must be. At the other end, taking $H$ trivial gives $m_\sigma=d_\sigma$ for every $\sigma$, recovering $L^2(X)=L^2(G)$ and the theory of the previous sections unchanged. Between these, enlarging $H$ shrinks $X$ and prunes the spectrum, and the pruning is not uniform across $\widehat{G}$; an irrep with no $H$-fixed vector is invisible on $X$ no matter what its ordering parameter is. In this sense, the pair $(G,H)$ selects which frequencies exist and the pair $(G,S)$ orders them, and the two choices are independent. The multiplicity $m_\sigma$ can exceed one, in which case the surviving copies of $\sigma$ all carry the same block of $\mathcal{L}$ and hence the same value $\omega(\sigma)$, so the ordering does not distinguish them. The case in which this degeneracy is absent is that of a Gelfand pair~\cite{diaconis1988group,ceccherini1945harmonic}, where $m_\sigma\le1$ for every $\sigma$ and $L^2(X)$ is multiplicity free. Each surviving irrep then occurs exactly once, in a $d_\sigma$-dimensional irreducible subspace, and the ordering parameter assigns a single value $\omega(\sigma)$ to that entire component.

\section{Examples}\label{sec:examples}

We now illustrate the general theory with four examples, two of which showcase the way our notion of smoothness depends on the choice of generating set. Moreover, once $X$ is a bare set, nothing selects the group acting on it, and different choices induce genuinely different notions of smoothness. We first discuss the latter point using two concrete examples. We then give two examples in which the generating set produces different orderings and therefore different notions of smoothness.

\subsection{Actions on the first n integers}
Let $X=\{0,\ldots,n-1\}$. One natural choice of action is $G=\mathbb{Z}_n$ acting by rotation; that is, $k\cdot x=x+k\mod n$. The action is simply transitive, so $H$ is trivial regardless of basepoint, and $m_\sigma=1$ for every $\sigma$. No frequency is lost in the descent to $X$. The identification $X\cong G$ sends $x$ to the group element $x_0-x$, and $L^2(X)\cong L^2(G)$ as spaces of functions.

Take $S=\{\pm1\}$, which is symmetric, generates $G$, and is closed under conjugation since $G$ is abelian. The Cayley graph is the $n$-cycle. All irreps are one-dimensional, given by the characters $\chi_k(x)=e^{2\pi ikx/n}$ for $k\in\mathbb{Z}_n$, so $\widehat{G}\cong\mathbb{Z}_n$, and each block of Proposition~\ref{prop:block-diag} is a single number. Explicitly, we have
\begin{equation}
    M_{\chi_k}=\frac{1}{2}\left(\chi_k(1)+\chi_k(-1)\right)=\cos(2\pi k/n),
\end{equation}
and since $d_{\chi_k}=1$, the trace is the value itself, giving
\begin{equation}
    \omega(\chi_k)=1-\cos(2\pi k/n).
\end{equation}
The ordering this induces is the familiar one. The trivial character $k=0$ has $\omega=0$ as required, and $\omega$ increases with $\lvert k\rvert$ measured cyclically, so the frequencies are ordered by how far $k$ sits from $0$ in $\mathbb{Z}_n$. Conjugate pairs $\chi_k$ and $\chi_{n-k}=\overline{\chi_k}$ receive equal values, consistent with Theorem~\ref{thm:image}. The spectrum is therefore symmetric and $\omega$ takes $\lfloor n/2\rfloor+1$ distinct values. The upper bound $\omega=2$ is attained exactly when $n$ is even by $k=n/2$, which is the alternating character $\chi_{n/2}(x)=(-1)^x$. This matches the bipartiteness criterion since the $n$-cycle is bipartite precisely for even $n$. 

Another natural choice of action is $G=\mathbb{S}_n$ acting by permutation; that is, $\pi\cdot x=\pi(x)$. The action is transitive but not simply so, and the stabilizer of any basepoint is the copy of $\mathbb{S}_{n-1}$ fixing that point, so $\lvert X\rvert=\frac{n!}{(n-1)!}=n$ as it must be. The frequencies visible on $X$ are those $\sigma$ with $m_\sigma>0$. Here $\mathbb{C}[X]$ carries the natural permutation representation of $\mathbb{S}_n$ on $n$ points, which decomposes as the trivial representation plus the standard representation, each once. Since $\dim(\mathcal{A}_\sigma\cap L^2(G)^H)=d_\sigma m_\sigma$ and $L^2(X)\cong\mathbb{C}[X]$, exactly two of the irreps of $\mathbb{S}_n$ survive, both with $m_\sigma=1$. Every other irrep, including all those of large dimension, is invisible on $X$ no matter what generating set is chosen.

The two surviving carrier spaces are the constant functions of dimension one and the mean-zero functions of dimension $n-1$. Since $m_\sigma=1$ for both, $(\mathbb{S}_n,\mathbb{S}_{n-1})$ is a Gelfand pair and the decomposition is multiplicity free. The trivial representation has $\omega=0$ for any choice of $S$. The standard representation receives a single value $\omega_\textnormal{std}>0$ whose size depends on $S$ but whose position in the ordering does not, since it is the only other frequency present.

The consequence is that the ordering on $X$ carries no information. Every $f\in L^2(X)$ splits uniquely as its mean plus a mean-zero remainder, and the only smoothness prior available is that $f$ be close to a constant. Contrasted with $\mathbb{Z}_n$ on the same $X$, where all $n$ frequencies survive, and the ordering distinguishes them, this shows that the meaning of smoothness is dependent on the choice of group action.

\subsection{Actions on bitstrings} Let $X=\{0,1\}^n$ and for a first action take $G=\mathbb{Z}_2^n$ acting by translation. With $x_0=(0,\ldots,0)$, the stabilizer $H$ is trivial, and we recover the $2^n$ frequencies ordered by Hamming weight studied by Belis et al.~\cite{belis2026spectral}. Indeed, we take $S=\{e_i\}_{i=1}^n$, where $e_i=(0,\ldots,0,1,0\ldots,0)$ denotes the standard basis vector with a one in the $i$-th component and a zero otherwise. The irreps are given by the characters $\chi_k(x)=(-1)^{k\cdot x}$ and we have
\begin{equation}
        (A\chi_k)(x)=\frac{1}{\lvert S\rvert}\sum_{s\in S}\chi_k(xs)
        =\chi_k(x)\frac{1}{\lvert S\rvert}\sum_{s\in S}\chi_k(s)
        =\chi_k(x)\frac{1}{n}\sum_{i=1}^n\chi_k(e_i)
        =\chi_k(x)\frac{1}{n}\sum_{i=1}^n(-1)^{k_i}.
\end{equation}
Letting $\lvert k\rvert$ denote the Hamming weight, this becomes
\begin{equation}
    (A\chi_k)(x)=\chi_k(x)\frac{1}{n}(n-2\lvert k\rvert).
\end{equation}
It follows that $\omega(\chi_k)=2\lvert k\rvert/n$, so that the ordering parameter reduces to a simple rescaling of the Hamming weight ordering.

A second action is given by the hyperoctahedral group $B_n=\mathbb{Z}_2^n\rtimes\mathbb{S}_n$, which permutes the coordinates in addition to the translation action. This action is also transitive on $X$ with $H=\mathbb{S}_n$ the stabilizer of $x_0=(0,\ldots,0)$. We again take $S=\{(e_i,\mathrm{id})\}_{i=1}^n$, which generates only the normal subgroup $\mathbb{Z}_2^n$. This is harmless because generation was assumed in Section~\ref{sec:finite-groups} only to make the Cayley graph connected, and $\mathbb{Z}_2^n$ already acts transitively on $X$, so the Schreier graph on $X$ is connected and $\omega$ still vanishes only at the trivial representation among the irreps visible on $X$. Restricting an irrep of $B_n$ to the normal subgroup $\mathbb{Z}_2^n$ decomposes it into characters, and $B_n$ acts on those characters through its quotient $\mathbb{S}_n$ by
\begin{equation}
    (\pi\cdot\chi_k)(x)=\chi_k(\pi^{-1}x)=\chi_{\pi(k)}(x),
\end{equation}
where $\pi$ permutes the coordinates of $k$. Since the orbits of $\mathbb{S}_n$ on $\mathbb{Z}_2^n$ are the Hamming levels, the $2^n$ characters of $\mathbb{Z}_2^n$ fuse into $n+1$ classes under the larger group, one for each value of $\lvert k\rvert$. The surviving irrep at level $j$ has carrier space $\operatorname{span}\{\chi_k:\lvert k\rvert=j\}$ of dimension $\binom{n}{j}$. That each such space is irreducible under $B_n$, and that no two are equivalent, is the statement that $(B_n,\mathbb{S}_n)$ is a Gelfand pair~\cite{ceccherini1945harmonic}, so every $m_\sigma$ is at most one and the decomposition of $L^2(X)$ is multiplicity free.

The comparison with the first action is the point. Passing from $\mathbb{Z}_2^n$ to $B_n$ does not change $X$ and does not change the values taken by the ordering parameter, since $\omega(\chi_k)=\frac{2\lvert k\rvert}{n}$ already depended on $k$ only through $\lvert k\rvert$. What changes is the resolution of the ordering. Under $\mathbb{Z}_2^n$, there are $2^n$ separate frequencies which happen to be assigned equal values in groups, while under $B_n$, those groups have merged into single frequencies of dimension $\binom{n}{j}$. The larger symmetry group removes the ability to distinguish coordinates, and the ordering becomes an ordering of $n+1$ objects rather than of $2^n$. This phenomenon is no accident, and the reason is elementary.

\begin{lemma}
    Let $\gamma$ be an automorphism of $G$ with $\gamma(S)=S$, and for an irrep $\sigma$ write $\sigma^\gamma=\sigma\circ\gamma$. Then $\omega(\sigma^\gamma)=\omega(\sigma)$.
\end{lemma}
\begin{proof}
    Since $\gamma$ restricts to a bijection of $S$, we have
    \begin{equation}
        M_{\sigma^\gamma}=\frac{1}{\lvert S\rvert}\sum_{s\in S}\sigma(\gamma(s))=\frac{1}{\lvert S\rvert}\sum_{s\in S}\sigma(s)=M_\sigma,
    \end{equation}
    and so the ordering parameters agree.
\end{proof}

\begin{figure*}
    \centering
    \begin{subfigure}[t]{0.48\textwidth}
        \includegraphics[width=\textwidth]{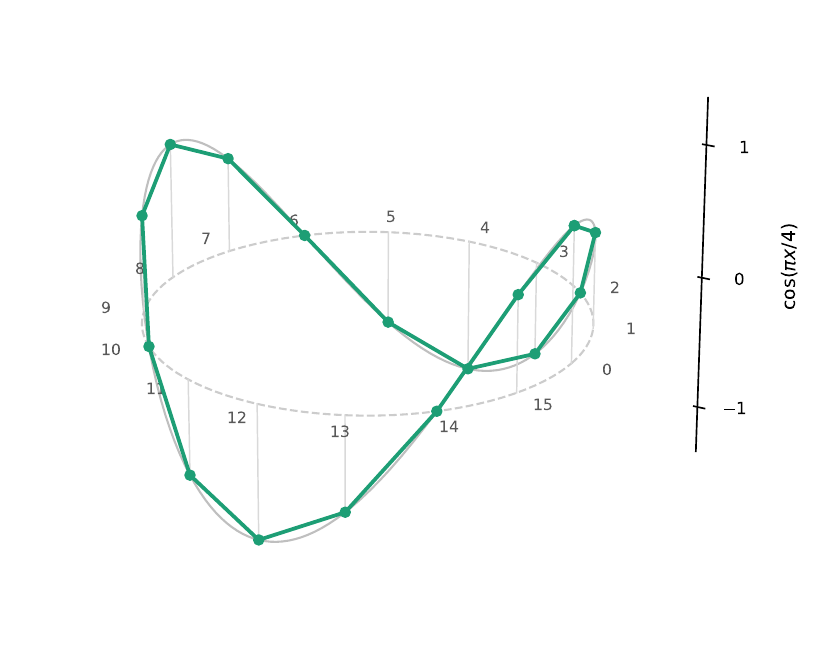}
    \end{subfigure}
    \begin{subfigure}[t]{0.48\textwidth}
        \includegraphics[width=\textwidth]{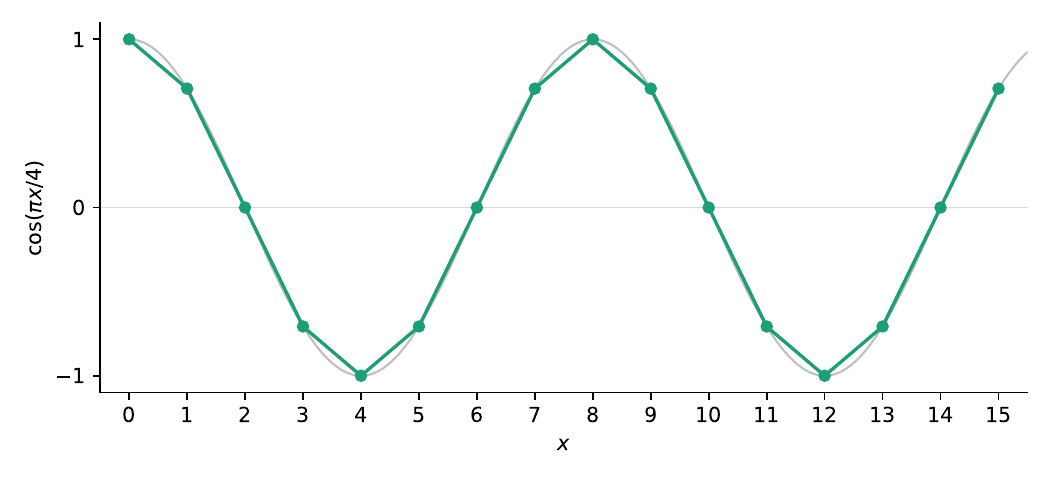}
    \end{subfigure}
    \caption{Plot of $f(x)=\frac12(\chi_2+\chi_{14})=\cos\left(\frac{\pi}{4}x\right)$ over the Cayley graph of $\mathbb{Z}_{16}$ with generating set $\{\pm1\}$. The plot above the actual Cayley graph is shown in the left panel, while the plot in a planar embedding is shown in the right panel.}
    \label{fig:s1-plots}
\end{figure*}

\subsection{Visualizing Smoothness on a Cyclic Group}
Consider the group $\mathbb{Z}_{16}$ whose irreps are the 16-th roots of unity $\chi_k(x)=e^{2\pi ikx/16}=e^{\pi ikx/8}$, and let us compare two different choices of generating set, $S_1=\{\pm1\}$ and $S_3=\{\pm3\}$. For $S_1$, the Cayley graph is given by the 16-cycle, but for visualization purposes, let us chop the edge between points 15 and 0 in half, and note that they are identified with each other. Thus, we may draw the Cayley graph as a straight line with 16 points labeled $0$ through $15$ with the caveat that when one reaches the end of the last edge, they loop back around to the edge entering $0$. This allows us to plot the values of a function on $\mathbb{Z}_{16}$ in a plane as one usually would. Consider the function $f(x)=\frac12(\chi_2(x)+\chi_{14}(x))=\cos\left(\frac{\pi}{4}x\right)$. Its plot along an axis orthogonal to the Cayley graph is pictured in Figure~\ref{fig:s1-plots} in both the circular Cayley graph format and the linear axis format.

For $S_3$, the Cayley graph is also given by the 16-cycle, but the group elements labeling each node have been permuted. Again plotting $f$, but with respect to this new ordering of the integers induced by the choice of generating set, we see that the traditionally smooth behavior of the cosine function disappears. The corresponding plots are shown in Figure~\ref{fig:s3-plots}.

\begin{figure*}
    \centering
    \begin{subfigure}[t]{0.48\textwidth}
        \includegraphics[width=\textwidth]{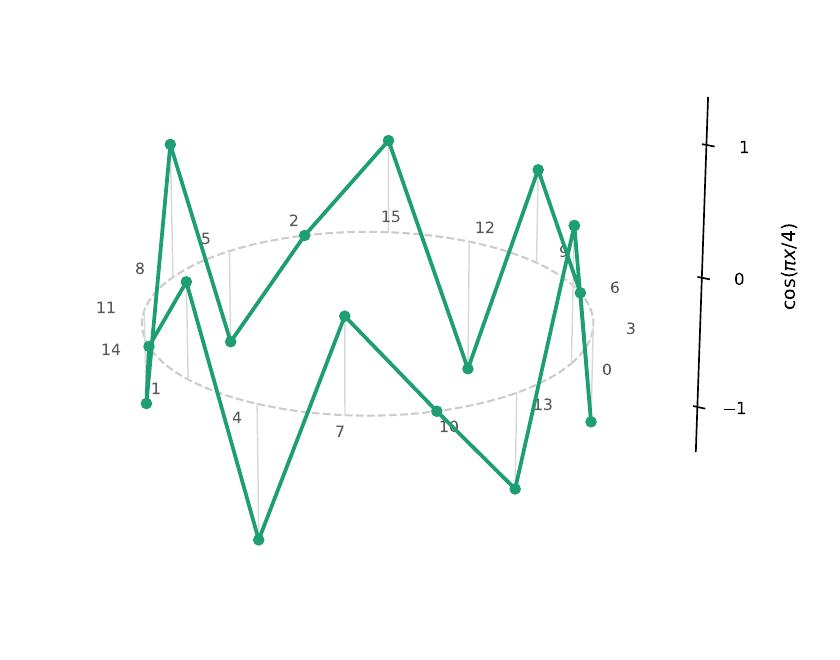}
    \end{subfigure}
    \begin{subfigure}[t]{0.48\textwidth}
        \includegraphics[width=\textwidth]{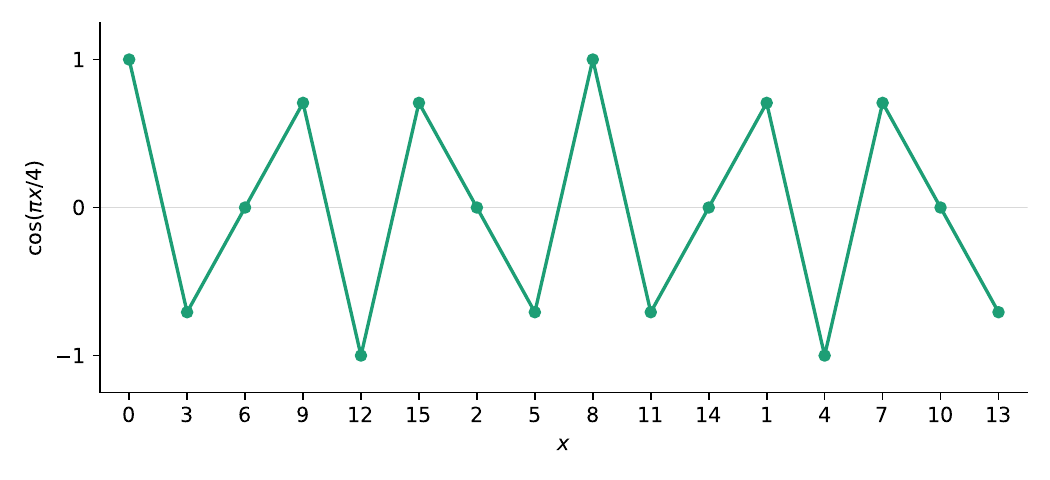}
    \end{subfigure}
    
    \caption{Plot of $f(x)=\frac12(\chi_2+\chi_{14})=\cos\left(\frac{\pi}{4}x\right)$ over the Cayley graph of $\mathbb{Z}_{16}$ with generating set $\{\pm3\}$. The plot above the actual Cayley graph is shown in the left panel, while the plot in a planar embedding is shown in the right panel.}
    \label{fig:s3-plots}
\end{figure*}

Comparing Figures~\ref{fig:s1-plots} and \ref{fig:s3-plots}, we see that $f$ appears more smooth when choosing the generating set $S_1$. Therefore, we expect that $f$ has Fourier coefficients concentrated at lower frequency when choosing $S_1$ for the generating set, and we expect $f$ to have Fourier coefficients concentrated at higher frequency when choosing $S_3$. Indeed, for $S_1$, the averaging operator is
\begin{equation}
    (Af)(x)=\frac12(f(x+1)+f(x-1)),
\end{equation}
and so
\begin{equation}
    (A\chi_k)(x)=\frac12\left(e^{2\pi ik(x+1)/16}+e^{2\pi ik(x-1)/16}\right)=\cos\left(\frac{\pi}{8}k\right)\chi_k(x)
\end{equation}
and the ordering is therefore
\begin{equation}
    \omega(\chi_k)=1-\cos\left(\frac{\pi}{8}k\right).
\end{equation}
Explicitly, we have $\omega(\chi_0)<\omega(\chi_1)=\omega(\chi_{15})<\omega(\chi_2)=\omega(\chi_{14})<\omega(\chi_3)=\omega(\chi_{13})<\omega(\chi_4)=\omega(\chi_{12})<\omega(\chi_5)=\omega(\chi_{11})<\omega(\chi_6)=\omega(\chi_{10})<\omega(\chi_7)=\omega(\chi_9)<\omega(\chi_8)$. Since $f=\frac12(\chi_2+\chi_{14})$, we see that $f$ is low frequency for $S_1$ as expected. For $S_3$, the averaging operator is
\begin{equation}
    (Af)(x)=\frac12(f(x+3)+f(x-3)),
\end{equation}
and so
\begin{equation}
    (A\chi_k)(x)=\frac12\left(e^{2\pi ik(x+3)/16}+e^{2\pi ik(x-3)/16}\right)=\cos\left(\frac{3\pi}{8}k\right)\chi_k(x)
\end{equation}
and the ordering is therefore
\begin{equation}
    \omega(\chi_k)=1-\cos\left(\frac{3\pi}{8}k\right).
\end{equation}
Explicitly, we have $\omega(\chi_0)<\omega(\chi_5)=\omega(\chi_{11})<\omega(\chi_6)=\omega(\chi_{10})<\omega(\chi_1)=\omega(\chi_{15})<\omega(\chi_4)=\omega(\chi_{12})<\omega(\chi_7)=\omega(\chi_{9})<\omega(\chi_2)=\omega(\chi_{14})<\omega(\chi_3)=\omega(\chi_{13})<\omega(\chi_8)$. Since $f=\frac12(\chi_2+\chi_{14})$, we see that $f$ is high frequency for $S_3$ as expected.

\subsection{Visualizing Smoothness on a Dihedral Group}
Consider the dihedral group $G=D_3$ and let $r$ and $f$ denote the rotation by 120 degrees and the flip about an axis through the triangle, respectively. This group has three irreps: the trivial representation defined by $\sigma_\textnormal{triv}(g)=1$, the sign representation defined by $\sigma_\textnormal{sgn}(r)=1$ and $\sigma_\textnormal{sgn}(f)=-1$, and the standard representation defined by
\begin{equation}
    \sigma_\textnormal{std}(r)=\frac12\begin{pmatrix}
        -1&-\sqrt{3}\\ \sqrt{3}&-1
    \end{pmatrix}
\end{equation}
and
\begin{equation}
    \sigma_\textnormal{std}(f)=\begin{pmatrix}
        1&0\\ 0&-1
    \end{pmatrix}.
\end{equation}
We again consider two generating sets, $S_1=\{r,f,r^2\}$ and $S_2=\{f,rf\}$. For $S_1$, the averaging operator is
\begin{equation}
    (Ah)(g)=\frac13(h(gr^2)+h(gr)+h(gf)),
\end{equation}
and so
\begin{equation}
    (A\sigma_\textnormal{triv})(g)=1=\sigma_\textnormal{triv},
\end{equation}
from which it follows that $\omega(\sigma_\textnormal{triv})=0$. Similarly, we have
\begin{equation}
    (A\sigma_\textnormal{sgn})(g)=\frac13(1+1-1)\sigma_\textnormal{sgn}(g)=\frac13\sigma_\textnormal{sgn}(g),
\end{equation}
from which it follows that $\omega(\sigma_\textnormal{sgn})=\frac23$. Finally, for the standard representation, we have
\begin{equation}
    (A\sigma_\textnormal{std})(g)=\frac13\sigma_\textnormal{std}(g)
    \begin{pmatrix}
        0&0\\0&-2
    \end{pmatrix},
\end{equation}
from which it follows that $\omega(\sigma_\textnormal{std})=\frac43$. So the ordering with respect to $S_1$ is $\omega(\sigma_\textnormal{triv})<\omega(\sigma_\textnormal{sgn})<\omega(\sigma_\textnormal{std})$.

Considering instead the generating set $S_2$, the averaging operator becomes
\begin{equation}
    (Ah)(g)=\frac12(h(gf)+h(grf)),
\end{equation}
and so
\begin{equation}
    (A\sigma_\textnormal{triv})(g)=1=\sigma_\textnormal{triv},
\end{equation}
from which it follows that $\omega(\sigma_\textnormal{triv})=0$. Similarly, we have
\begin{equation}
    (A\sigma_\textnormal{sgn})(g)=-\sigma_\textnormal{sgn}(g),
\end{equation}
from which it follows that $\omega(\sigma_\textnormal{sgn})=2$. Finally, for the standard representation, we have
$\omega(\sigma_\textnormal{std})=1$. So the ordering with respect to $S_2$ is $\omega(\sigma_\textnormal{triv})<\omega(\sigma_\textnormal{std})<\omega(\sigma_\textnormal{sgn})$.

\begin{figure*}
    \includegraphics[width=\textwidth]{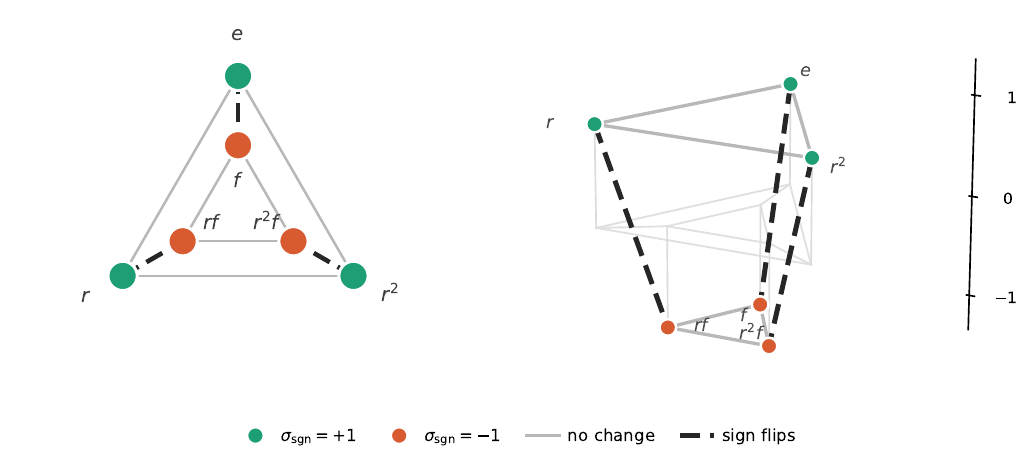}
    \caption{The sign representation of $D_3$ on the Cayley graph of $(D_3,S_1)$ with $S_1=\{r,r^2,f\}$, shown in the plane and lifted to the value of the function. Dashed edges are those across which the sign changes. Three of the nine edges are dashed, and $\omega(\sigma_\textnormal{sgn})=2/3$.}
    \label{fig:s1-dihedral-plots}
\end{figure*}

We plot the function $\sigma_\textnormal{sgn}$ over the Cayley graph in Figures~\ref{fig:s1-dihedral-plots} and \ref{fig:s2-dihedral-plots} for both generating sets. For $S_1$, the sign representation is the second frequency, whereas for $S_2$, it is the third. Thus, the sign function should be interpreted as smoother in the former case. Indeed, looking at the figures, the function value changes dramatically less often in the case of $S_1$ than it does in the case of $S_2$. In fact, for $S_2$, every edge in the Cayley graph corresponds to a maximal amount of change.
 
\begin{figure*}
    \includegraphics[width=\textwidth]{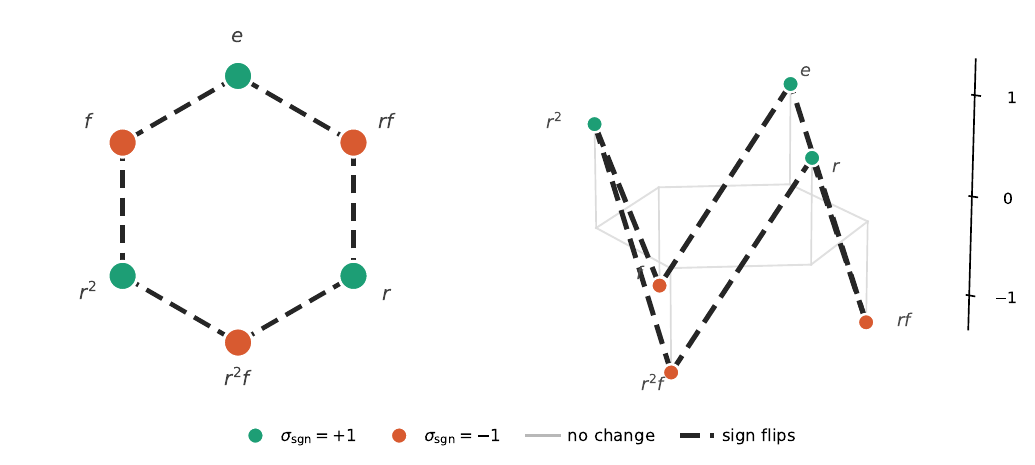}
    \caption{The sign representation of $D_3$ on the Cayley graph of $(D_3,S_2)$ with $S_2=\{f,rf\}$, shown in the plane and lifted to the value of the function. Dashed edges are those across which the sign changes. Every generator is a reflection, so every edge joins a rotation to a reflection and all six edges are dashed, giving the maximal value $\omega(\sigma_\textnormal{sgn})=2$.}
    \label{fig:s2-dihedral-plots}
\end{figure*}

\section{Smoothness as an Inductive Bias}\label{sec:prior}

Suppose we are not analyzing a function, but attempting to build one. Training an interpretable machine learning model is the prototypical example~\cite{gilpin2018explaining,murdoch2019definitions,gili2026inherent}. We wish to bias the construction toward smooth functions, but this presupposes a choice of group action and generating set and begs the question of how these objects should be chosen, given a problem. What are the features of the problem that make one choose group $G$ over group $H$ or generating set $S$ over generating set $T$? Let us fix a problem and explore how a designer might settle these questions. 

An RNA molecule is a string over the four letters $\mathtt{A}$, $\mathtt{C}$, $\mathtt{G}$, and $\mathtt{U}$. Unlike DNA, it is single stranded, so it folds back on itself. A letter at one position can pair with a complementary letter elsewhere in the same string, $\mathtt{A}$ with $\mathtt{U}$, $\mathtt{G}$ with $\mathtt{C}$, and $\mathtt{G}$ with $\mathtt{U}$. A \emph{secondary structure} is a set of such pairs with each position occurring in at most one pair. Two common additional assumptions are made: a minimum hairpin length so that paired positions $i<j$ satisfy $j-i>3$, and the absence of pseudoknots so that any two pairs $(i,j)$ and $(k,l)$ with $i<k$ satisfy either $j<k$ or $l<j$. Each such structure is assigned a free energy by the standard nearest-neighbor thermodynamic model~\cite{turner2010nndb}, in which stacked pairs lower the energy and the unpaired loops they enclose raise it by an amount growing with the length of the loop (see Figure~\ref{fig:stack}). The molecule occupies low energy structures, and the minimum free energy over all pseudoknot-free secondary structures of a given string, which we write $E(\textnormal{string})$, is computable by dynamic programming~\cite{zuker1981optimal} in time cubic in the length once internal loops are capped at the conventional size. Dropping the pseudoknot-free restriction changes the problem, and minimization over general pseudoknotted structures is NP-hard for standard energy models~\cite{lyngso2000pseudoknots}.

\begin{figure}
    \centering
    \includegraphics[width=0.5\linewidth]{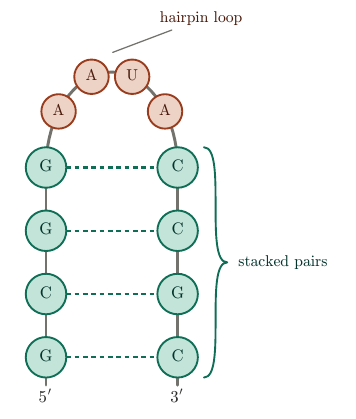}
    \caption{A hairpin, the simplest pseudoknot-free secondary structure. Dashed rungs are base pairs and the unpaired letters at the top form the loop closed by the stem. Each adjacent pair of base pairs is a stack that lowers the free energy, while the enclosed loop raises it by an amount growing with its length.}
    \label{fig:stack}
\end{figure}

Now suppose $n=12$ short RNA blocks called \emph{modules} are fixed in advance, and a construct is built by concatenating all twelve in some order. The designer wants the arrangement whose fold is most stable, so the response is $f(\pi)=E(\text{concatenation in the order }\pi)$ and the goal is to minimize it over $X$. Each evaluation is one dynamic program and is cheap, but $\lvert X\rvert=12!=479001600$, so enumeration is out of reach and the designer budgets $500$ evaluations. The task is to fit a surrogate to those $500$ values and minimize the surrogate exactly.

The domain and the group are settled by asking what leaves $f$ unchanged. If two modules had identical sequences, exchanging their labels would be an invariance, the domain would be the quotient of $X$ by that exchange, and Proposition~\ref{prop:restricted-fourier} would delete the frequencies with no fixed vector. An RNA string has a direction and the reverse of a construct is a different molecule with a different fold, not the same one relabeled, so the one symmetry a designer might hope for beyond the above is unavailable. With twelve distinct modules and no reversal symmetry, we record the full order, take $X=G=\mathbb{S}_{12}$, and let it act on itself by translation. This is the smallest transitive choice and the stabilizer is trivial. Moreover, $m_\lambda=\chi_\lambda(e)=d_\lambda$ for every partition $\lambda\vdash 12$, so nothing is deleted.

When choosing the generating set, our goal is to choose group elements which can be naturally considered incremental steps on our data. Two such generating sets come to mind. The first is the adjacent transpositions
\begin{equation}
    S_\textnormal{adj}=\{(i\ \ i+1):1\le i\le n-1\},
\end{equation}
which interchange two neighbors and leave every other item alone. Indeed, in the context of a ranking, such a transposition performs the minimal change by swapping the ordering of two adjacent items. The Cayley graph corresponding to this choice of generator is the permutohedron, and the generalized Hamming weight of Section~\ref{sec:finite-groups} is the number of inversions, which we note is Kendall's tau distance~\cite{kendall1938new}. The second generating set is the insertions
\begin{equation}
    S_\textnormal{ins}=\{(i\ i+1\ \cdots\ j)^{\pm1}:1\le i<j\le n\}
\end{equation}
which are obtained by removing an item in position $j$ and reinserting it in position $i$, sliding the intervening items over by one. The size of this set is $\lvert S_\textnormal{ins}\rvert=(n-1)^2$, with $n-1$ moves of cycle length two and $2(n-m+1)$ of cycle length $m$ for each $3\le m\le n$. Insertion is the elementary edit of a sorted list and the move underlying insertion sort, and it is no less local a description of small change than a neighbor swap; the two notions are simply different and should be applied in different contexts. 

In the context of our RNA secondary structure problem, the generating set is chosen by asking which rearrangements the designer is willing to call small. A stem (the stack of dashed rungs in Figure~\ref{fig:stack}) formed between two modules closes a loop containing everything between its arms, and the penalty for that loop grows with how far apart the two modules sit on the chain. An adjacent transposition changes the number of intervening modules between any pair by at most one, and since the modules are fixed, it changes any separation measured in nucleotides by at most twice the length of the longest module. An insertion moves a module across up to eleven others and can change the separations involving it by a distance comparable to the length of most of the construct. Holding the pairing set fixed, adjacent transpositions therefore perturb loop lengths on a scale set by individual module lengths, whereas insertions can perturb them on the scale of the whole construct. What that argument bounds is the change in the separations, rather than the change in the energy, since the minimizing structure itself moves with the order and the stacking and junction terms depend on the letters adjacent to each pair; a single swap can create or destroy a stem, and an insertion that preserves the relative order of a module and its partners can leave the energy nearly fixed. Here, we take $S=S_\textnormal{adj}$ as a declaration of which rearrangements count as incremental.

The designer may further believe that the modules near the ends of the construct matter less than those in the middle. After all, the ends have fewer ways to be enclosed. That belief cannot be expressed here. All adjacent transpositions lie in the single conjugacy class $K$ of transpositions, so a weighting $\varphi=\lambda_e\delta_e-\sum_iw_i\delta_{s_i}$ with $w_i\ge0$ and $\lambda_e=\sum_iw_i$, the latter being the condition that $T_\varphi$ annihilate the constants, has conjugation average $\varphi^\natural=\sum_iw_i(\delta_e-\mu_K)$, and by Proposition~\ref{prop:admissible-block}, the induced ordering is a nonnegative multiple of $\omega_K$ no matter what the individual $w_i$ are, and a positive one as soon as some $w_i>0$. Position dependent weights change the operator but not the ordering, so the belief has to enter through the model rather than the prior.

With these choices made, the designer chooses a bandlimit to enforce smoothness with, and this is settled by counting. Every generator is a transposition, so $\omega(\lambda)=1-\chi_\lambda(\tau)/d_\lambda$ for a single transposition $\tau$, and Frobenius' formula~\cite{diaconis1988group} gives
\begin{equation}
    \omega(\lambda)=1-\frac{2}{n(n-1)}\sum_{(i,j)\in\lambda}(j-i),
\end{equation}
the sum of the contents of the cells of $\lambda$. For $n=12$, the smallest values are $\omega((12))=0$, then $\omega((11,1))=\frac{2}{11}$, then $\omega((10,2))=\frac13$. Since $m_\lambda=d_\lambda$, the parameter count below a cutoff is $\sum_\lambda d_\lambda^2$. A cutoff at $\frac{2}{11}$ gives $1+11^2=122$ parameters and a cutoff at $\frac13$ gives $122+54^2=3038$. With $500$ evaluations, the first is estimable and the second is not, so the cutoff is chosen to be $\frac{2}{11}$.

We now identify the 122 dimensional space concretely, rather than as an abstract isotypic sum. Consider the permutation matrices $P(\pi)$ with $(j,i)$ entry $\delta_{\pi(i),j}$. This matrix is $12\times12$, and so we have $144$ functions $\pi\mapsto\delta_{\pi(i),j}$. The map $P$ is the natural permutation representation on $12$ points, which decomposes as $(12)\oplus(11,1)$, and so every matrix entry $\delta_{\pi(i),j}$ lies in the span of the two surviving blocks after the cutoff. Conversely, the span is all of it, and we can check the dimension by hand. We ask which matrices $A=(a_{ij})$ give
\begin{equation}
    \sum_{i,j}a_{ij}\delta_{\pi(i),j}=\sum_ia_{i,\pi(i)}=0
\end{equation}
for every permutation $\pi$. Comparing two permutations that differ by a transposition forces $a_{ij}=r_i+c_j$ for some constants $r_i$ and $c_j$, and it follows that
\begin{equation}\label{eq:sum-condition}
    \sum_ir_i=-\sum_ic_{\pi(i)}=-\sum_i c_i.
\end{equation}
The pairs $(r,c)$ give 24 parameters, minus 1 for the shift $(r+t,c-t)$ that leaves $A$ unchanged, and minus 1 for the sum condition \eqref{eq:sum-condition}. That leaves a 22-dimensional kernel, so the span has dimension $144-22=122$ as expected.

What survives is $\operatorname{span}\{\pi\mapsto\delta_{\pi(i),j}\}$, the first order model of ranked data~\cite{diaconis1988group,chen2021signal}, in which the surrogate is
\begin{equation}
    \widetilde{f}(\pi)=\sum_{i,j}a_{ij}\delta_{\pi(i),j}
\end{equation}
for a matrix of module-position energies. The declaration has done two things at once. It has cut the parameter count from $479001600$ to $122$, which is what makes the fit possible from one budget of folds, and it has left an objective whose minimization over $\mathbb{S}_{12}$ is a linear assignment problem and therefore exactly solvable in cubic time rather than by search. The four choices were decided by four different features of the problem. What the energy is invariant under fixed the domain and group, the designer's declaration of which rearrangements count as incremental fixed the generating set, the conjugacy structure of that set determined which further beliefs the prior could carry, and the evaluation budget fixed the bandlimit. Together, these choices produce an inductive bias which simplifies the problem in much the same way a symmetry constraint does in geometric deep learning~\cite{bronstein2017geometric,bronstein2021geometric,gerken2023geometric} and perhaps more explicitly so its quantum counterpart known as geometric quantum machine learning~\cite{meyer2023exploiting,larocca2022group,bradshaw2025learning}.

\section{Smoothness on Compact Groups}\label{sec:compact}
The construction of Section~\ref{sec:finite-groups} used finiteness of $G$ in exactly two places: to normalize the counting measure and to draw the Cayley graph. The first is inessential and the second is not, and separating the two clarifies what the ordering parameter really depends on. In this section, we let $G$ be a compact Hausdorff group with normalized Haar measure $dg$, and we write $L^2(G)$ for the associated Hilbert space with inner product
\begin{equation}
    \langle f,h\rangle=\int_G\overline{f(g)}h(g)\ dg.
\end{equation}

The Peter--Weyl theorem gives an orthogonal decomposition
\begin{equation}
    L^2(G)=\bigoplus_{\sigma\in\widehat{G}}\mathcal{A}_\sigma,
\end{equation}
where $\mathcal{A}_\sigma\cong\mathcal{H}_\sigma\otimes\mathcal{H}^*_\sigma$ and $\mathcal{H}_\sigma$ is the carrier space for the irrep $\sigma$. The decomposition is an orthogonal Hilbert direct sum over $\widehat{G}$, with each $d_\sigma$ finite, and the matrix entries $\sqrt{d_\sigma}\sigma_{ij}$ again form an orthonormal basis. Fix a finite symmetric set $S\subset G$ with $e\notin S$ and define the averaging operator $A$ and the Laplacian $\mathcal{L}$ as we did before. Each right translation is unitary on $L^2(G)$, so $A$ is a convex combination of unitaries and $\mathcal{L}$ is bounded with $\| A\|\le1$. Both of the following basic results carry over without much modification as neither proof uses finiteness.

\begin{proposition}\label{prop:quadratic-compact}
    Let $f\in L^2(G)$. Then
    \begin{equation}
        \langle f,\mathcal{L}f\rangle=\frac{1}{2\lvert S\rvert}\sum_{s\in S}\int_G\lvert f(g)-f(gs)\rvert^2\ dg.
    \end{equation}
\end{proposition}

\begin{proposition}\label{prop:block-compact}
    The Laplacian $\mathcal{L}$ preserves each $\mathcal{A}_\sigma$ and acts there as $I_{d_\sigma}\otimes(I_{d_\sigma}-M_\sigma)$, where $M_\sigma=\frac{1}{\lvert S\rvert}\sum_{s\in S}\sigma(s)$. Consequently, $\omega(\sigma)=1-\operatorname{Tr}[M_\sigma]/d_{\sigma}$ is defined for every $\sigma\in\widehat{G}$ and satisfies $\omega(\sigma)\in[0,2]$.
\end{proposition}

The proof of Proposition~\ref{prop:quadratic-compact} is that of Proposition~\ref{prop:quadratic} with the sum over $G$ replaced by the Haar integral, using right invariance of $dg$ in place of the re-indexing of a finite sum. The proof of Proposition~\ref{prop:block-compact} is that of Proposition~\ref{prop:block-diag} verbatim, since it uses only $\sigma$ as a finite-dimensional unitary homomorphism. The ordering parameter is therefore available for every compact group, and it is again an invariant of the pair $(G,S)$.

What does not carry over is the Cayley graph picture. For continuous $G$, the vertex set can be uncountable, and more seriously, the graph is never connected when $G$ is infinite. Indeed, a compact Hausdorff group is homogeneous, so it has no isolated points unless it is discrete, and a countable space without isolated points is not Baire; hence an infinite compact Hausdorff group is uncountable. A finitely generated subgroup is countable and therefore proper. The components are the cosets of $\langle S\rangle$, and the word metric is finite only within a component. The graph is thus a device internal to $G$ rather than a combinatorial model of $G$, and the geometric reading in which nearby group elements differ by few generators is not available. What replaces it is Proposition~\ref{prop:quadratic-compact}. The quadratic form survives intact as an integral, and it is this functional rather than the graph that the ordering measures.

Now suppose in addition that $G$ is a compact connected Lie group. A different limit recovers a geometric picture. Replacing the finite set $S$ by a family of conjugation invariant probability measures concentrated near the identity and imposing the usual diffusive scaling and second-moment hypotheses produces a family of averaging operators whose associated Laplacians converge, after suitable rescaling, to the Laplace--Beltrami operator~\cite{applebaum2014probability,liao2004levy} of an Ad-invariant metric on $G$; that is, to the Casimir element. For $G$ connected with simple Lie algebra, such a metric is unique up to scale, so the induced frequency ordering is canonical up to an overall positive scale. That limiting ordering is the familiar one. The Casimir eigenvalue is the weight against which the Sobolev scale on $\widehat{G}$ is built~\cite{ruzhansky2010pseudo}, and smoothness in the classical sense is equivalent to superpolynomial decay of the Fourier coefficients measured against it~\cite{sugiura1971fourier}. The scaling limit is also the diffusion half of Hunt's classification of the generators of convolution semigroups on a Lie group~\cite{hunt1956semi}, the jump half of which is the continuous counterpart of the weighted Cayley-Laplacians of Proposition~\ref{prop:positivity}. In particular, when $G$ has simple Lie algebra, the dependence on a generating set disappears, up to an overall positive scale, in this conjugation-invariant infinitesimal limit.

\section{Conclusion}\label{sec:conclusion}
We have argued that smoothness on $L^2(G)$ is not a property a function possesses on its own. A function is smooth relative to a declaration of which group elements count as a small step, and once that declaration is made, the Cayley-Laplacian construction determines an ordering. The declaration is a symmetric generating set $S$, the structure it induces is the Cayley graph of $(G,S)$, and the roughness of a function is the Dirichlet energy of that graph. Because the graph is a Cayley graph, its Laplacian is block diagonal over the dual, and the mean of the eigenvalues in the block at $\sigma$ is a single number $\omega(\sigma)$ attached to the irrep rather than to a choice of basis. Ordering $\widehat{G}$ by $\omega$ is what makes the phrase ``concentrated at low frequency'' meaningful for a non-abelian group.

Section~\ref{sec:characterization} measures how much of this was forced. The four requirements, that the operator be Hermitian, that it distinguish no point of $G$, that it assign zero energy to the constant functions, and that the energy it assigns not depend on a choice of complex conjugate, turn out to constrain the ordering only up to reality, agreement on conjugate pairs, and vanishing at the trivial representation. Every function on $\widehat{G}$ with those properties is realized by an admissible operator, and the orderings $\omega_K$ coming from inversion orbits of conjugacy classes are a basis for them. So these axioms should be read as supplying coordinates rather than a restriction; an ordering becomes a weight vector on $\mathcal{K}$, and a symmetric generating set is the case in which that vector is the probability distribution induced on the classes the set meets. The restrictions come afterward. Nonnegativity of the weights identifies a natural positive semidefinite subclass, namely the weighted Cayley-Laplacians, among all admissible operators. Adding locality and uniformity with respect to a fixed generating set leaves only the Laplacian of Section~\ref{sec:finite-groups} up to positive scale.

Two further layers of dependence emerged along the way. When the domain is a set rather than a group, the acting group determines which frequencies exist at all, since an irrep is visible precisely when it contains a vector fixed by the stabilizer, while the generating set orders whatever survives. These choices are independent, and the examples of Section~\ref{sec:examples} show that they can be made to disagree sharply, with $\mathbb{Z}_n$ and $\mathbb{S}_n$ acting on the same $n$ labels producing an ordering of $n$ frequencies in one case and of two in the other. Within a fixed group, changing the generating set can invert the ordering outright, as the sign representation of $D_3$ illustrates by moving from second to last.

Several directions remain. The ordering parameter is a mean over a block, and when $S$ is not closed under conjugation, the block genuinely splits, so the ordering on $\widehat{G}$ is strictly coarser than the ordering on the spectrum of $\mathcal{L}$. For compact groups, we established that the block diagonalization and the bounds persist, but the characterization theorem does not transfer directly, since its proof uses finiteness both in decomposing an invariant function over the partition $G$ by conjugacy classes and in the evaluation argument establishing linear independence. A compact analogue replacing $\mathcal{K}$ by conjugation invariant measures seems plausible and would complete the picture. Hunt's theorem indicates the form that analogue should take, with a Casimir term in place of the diagonal of $\varphi$ and a conjugation invariant L\'evy measure in place of the weights $\alpha_K$~\cite{hunt1956semi,applebaum2011infinitely}.

\bibliographystyle{unsrt}
\bibliography{refs.bib}

@article{belis2026spectral,
  title={Spectral methods: crucial for machine learning, natural for quantum computers?},
  author={Belis, Vasilis and Bowles, Joseph and Gupta, Rishabh and Peters, Evan and Schuld, Maria},
  journal={arXiv preprint arXiv:2603.24654},
  year={2026}
}

@article{bradshaw2026symmetry,
  title={Symmetry-based quantum codes beyond the {P}auli group},
  author={Bradshaw, Zachary P and LaBorde, Margarite L and Montero, Dillon},
  journal={Physical Review A},
  volume={113},
  number={4},
  pages={042431},
  year={2026},
  publisher={APS}
}

@article{rethinasamy2024logarithmic,
  title={Logarithmic-depth quantum circuits for {H}amming weight projections},
  author={Rethinasamy, Soorya and LaBorde, Margarite L and Wilde, Mark M},
  journal={Physical Review A},
  volume={110},
  number={5},
  pages={052401},
  year={2024},
  publisher={APS}
}

@article{gili2026inherent,
  title={Inherent interpretability provides inherent value in quantum machine learning},
  author={Gili, Kaitlin and Bradshaw, Zachary P},
  journal={arXiv preprint arXiv:2607.13827},
  year={2026}
}

@book{diaconis1988group,
  title={Group representations in probability and statistics},
  author={Diaconis, Persi},
  volume={11},
  year={1988},
  publisher={Institute of mathematical statistics Hayward, CA}
}

@book{terras1999fourier,
  title={Fourier analysis on finite groups and applications},
  author={Terras, Audrey},
  number={43},
  year={1999},
  publisher={Cambridge University Press}
}

@book{chung1997spectral,
  title={Spectral graph theory},
  author={Chung, Fan RK},
  volume={92},
  year={1997},
  publisher={American Mathematical Soc.}
}

@article{diaconis1981generating,
  title={Generating a random permutation with random transpositions},
  author={Diaconis, Persi and Shahshahani, Mehrdad},
  journal={Zeitschrift f{\"u}r Wahrscheinlichkeitstheorie und verwandte Gebiete},
  volume={57},
  number={2},
  pages={159--179},
  year={1981},
  publisher={Springer-Verlag Berlin/Heidelberg}
}

@article{shuman2013emerging,
  title={The emerging field of signal processing on graphs: Extending high-dimensional data analysis to networks and other irregular domains},
  author={Shuman, David I and Narang, Sunil K and Frossard, Pascal and Ortega, Antonio and Vandergheynst, Pierre},
  journal={IEEE signal processing magazine},
  volume={30},
  number={3},
  pages={83--98},
  year={2013},
  publisher={IEEE}
}

@article{babai1979spectra,
  title={Spectra of {C}ayley graphs},
  author={Babai, L{\'a}szl{\'o}},
  journal={Journal of Combinatorial Theory, Series B},
  volume={27},
  number={2},
  pages={180--189},
  year={1979},
  publisher={Elsevier}
}

@article{caputo2010proof,
  title={Proof of {A}ldous’ spectral gap conjecture},
  author={Caputo, Pietro and Liggett, Thomas and Richthammer, Thomas},
  journal={Journal of the American Mathematical Society},
  volume={23},
  number={3},
  pages={831--851},
  year={2010}
}

@article{ortega2018graph,
  title={Graph signal processing: Overview, challenges, and applications},
  author={Ortega, Antonio and Frossard, Pascal and Kova{\v{c}}evi{\'c}, Jelena and Moura, Jos{\'e} MF and Vandergheynst, Pierre},
  journal={Proceedings of the IEEE},
  volume={106},
  number={5},
  pages={808--828},
  year={2018},
  publisher={IEEE}
}

@article{leus2023graph,
  title={Graph signal processing: History, development, impact, and outlook},
  author={Leus, Geert and Marques, Antonio G and Moura, Jos{\'e} MF and Ortega, Antonio and Shuman, David I},
  journal={IEEE Signal Processing Magazine},
  volume={40},
  number={4},
  pages={49--60},
  year={2023},
  publisher={IEEE}
}

@article{dong2020graph,
  title={Graph signal processing for machine learning: A review and new perspectives},
  author={Dong, Xiaowen and Thanou, Dorina and Toni, Laura and Bronstein, Michael and Frossard, Pascal},
  journal={IEEE Signal processing magazine},
  volume={37},
  number={6},
  pages={117--127},
  year={2020},
  publisher={IEEE}
}

@book{fulton2013representation,
  title={Representation theory: a first course},
  author={Fulton, William and Harris, Joe},
  year={2013},
  publisher={Springer Science \& Business Media}
}

@book{rudin1962fourier,
  title={Fourier analysis on groups},
  author={Rudin, Walter},
  volume={12},
  year={1962},
  publisher={Interscience publishers New York}
}

@article{ceccherini1945harmonic,
  title={Harmonic analysis on finite groups},
  author={Ceccherini-Silberstein, Tullio and Scarabotti, Fabio and Tolli, Filippo},
  journal={Cambridge studies in advanced mathematics},
  volume={108},
  year={2008}
}

@book{odonnel2014analysis,
  title={Analysis of boolean functions},
  author={O'Donnell, Ryan},
  volume={2},
  year={2014},
  publisher={Cambridge University Press Cambridge}
}

@book{applebaum2014probability,
  title={Probability on compact Lie groups},
  author={Applebaum, David},
  year={2014},
  publisher={Springer}
}

@book{liao2004levy,
  title={L{\'e}vy processes in Lie groups},
  author={Liao, Ming},
  volume={162},
  year={2004},
  publisher={Cambridge university press}
}

@inproceedings{rahaman2019spectral,
  title={On the spectral bias of neural networks},
  author={Rahaman, Nasim and Baratin, Aristide and Arpit, Devansh and Draxler, Felix and Lin, Min and Hamprecht, Fred and Bengio, Yoshua and Courville, Aaron},
  booktitle={International conference on machine learning},
  pages={5301--5310},
  year={2019},
  organization={PMLR}
}

@article{xu2019frequency,
  title={Frequency principle: Fourier analysis sheds light on deep neural networks},
  author={Xu, Zhi-Qin John and Zhang, Yaoyu and Luo, Tao and Xiao, Yanyang and Ma, Zheng},
  journal={arXiv preprint arXiv:1901.06523},
  year={2019}
}

@article{rockmore2002fast,
  title={Fast Fourier transform for fitness landscapes},
  author={Rockmore, Dan and Kostelec, Peter and Hordijk, Wim and Stadler, Peter F},
  journal={Applied and Computational Harmonic Analysis},
  volume={12},
  number={1},
  pages={57--76},
  year={2002},
  publisher={Elsevier}
}

@inproceedings{ghandehari2019non,
  title={A non-commutative viewpoint on graph signal processing},
  author={Ghandehari, Mahya and Guillot, Dominique and Hollingsworth, Kristopher},
  booktitle={2019 13th International conference on Sampling Theory and Applications (SampTA)},
  pages={1--4},
  year={2019},
  organization={IEEE}
}

@article{beck2024frames,
  title={Frames for signal processing on Cayley graphs},
  author={Beck, Kathryn and Ghandehari, Mahya and Hudson, Skyler and Paltenstein, Jenna},
  journal={Journal of Fourier Analysis and Applications},
  volume={30},
  number={6},
  pages={66},
  year={2024},
  publisher={Springer}
}

@article{chen2021signal,
  title={Signal processing on the permutahedron: tight spectral frames for ranked data analysis},
  author={Chen, Yilin and DeJong, Jennifer and Halverson, Tom and Shuman, David I},
  journal={Journal of Fourier Analysis and Applications},
  volume={27},
  number={4},
  pages={70},
  year={2021},
  publisher={Springer}
}

@article{stadler1996landscapes,
  title={Landscapes and their correlation functions},
  author={Stadler, Peter F},
  journal={Journal of Mathematical chemistry},
  volume={20},
  number={1},
  pages={1--45},
  year={1996},
  publisher={Springer}
}

@article{hordijk1998amplitude,
  title={Amplitude spectra of fitness landscapes},
  author={Hordijk, Wim and Stadler, Peter F},
  journal={Advances in Complex Systems},
  volume={1},
  number={01},
  pages={39--66},
  year={1998},
  publisher={World Scientific}
}

@article{azangulov2024stationary,
  title={Stationary kernels and Gaussian processes on Lie groups and their homogeneous spaces I: the compact case},
  author={Azangulov, Iskander and Smolensky, Andrei and Terenin, Alexander and Borovitskiy, Viacheslav},
  journal={Journal of Machine Learning Research},
  volume={25},
  number={280},
  pages={1--52},
  year={2024}
}

@incollection{stadler2002fitness,
  title={Fitness landscapes},
  author={Stadler, Peter F},
  booktitle={Biological evolution and statistical physics},
  pages={183--204},
  year={2002},
  publisher={Springer}
}

@book{isaacs1994character,
  title={Character theory of finite groups},
  author={Isaacs, I Martin},
  volume={69},
  year={1994},
  publisher={Courier Corporation}
}

@article{ruzhansky2010pseudo,
  title={Pseudo-differential operators and symmetries},
  author={Ruzhansky, Michael and Turunen, Ville},
  journal={Background analysis and advanced topics, Pseudo-Differential Operators. Theory and Applications},
  volume={2},
  year={2010},
  publisher={Birkh{\"a}user Verlag Basel}
}

@article{sugiura1971fourier,
  author={Sugiura, Mitsuo},
  title={Fourier series of smooth functions on compact Lie groups},
  journal={Osaka Journal of Mathematics}, 
  volume={8}, 
  pages={33--47}, 
  year={1971}
  }

@article{hunt1956semi,
  title={Semi-groups of measures on Lie groups},
  author={Hunt, Gilbert Agnew},
  journal={Transactions of the American Mathematical Society},
  volume={81},
  number={2},
  pages={264--293},
  year={1956}
}

@article{applebaum2011infinitely,
  author={Applebaum, David},
  title={Infinitely divisible central probability measures on compact Lie groups---regularity, semigroups and transition kernels},
  journal={The Annals of Probability}, 
  volume={39}, 
  number={6}, 
  pages={2474--2496}, 
  year={2011}
  }

@article{kendall1938new,
  title={A new measure of rank correlation},
  author={Kendall, Maurice G and others},
  journal={Biometrika},
  volume={30},
  number={1-2},
  pages={81--93},
  year={1938},
  publisher={Oxford University Press}
}

@article{turner2010nndb,
  title={NNDB: the nearest neighbor parameter database for predicting stability of nucleic acid secondary structure},
  author={Turner, Douglas H and Mathews, David H},
  journal={Nucleic acids research},
  volume={38},
  number={suppl\_1},
  pages={D280--D282},
  year={2010},
  publisher={Oxford University Press}
}

@article{zuker1981optimal,
  title={Optimal computer folding of large RNA sequences using thermodynamics and auxiliary information},
  author={Zuker, Michael and Stiegler, Patrick},
  journal={Nucleic acids research},
  volume={9},
  number={1},
  pages={133--148},
  year={1981},
  publisher={Oxford University Press}
}

@inproceedings{lyngso2000pseudoknots,
  title={Pseudoknots in RNA secondary structures},
  author={Lyngs{\o}, Rune B and Pedersen, Christian NS},
  booktitle={Proceedings of the fourth annual international conference on Computational molecular biology},
  pages={201--209},
  year={2000}
}

@inproceedings{gilpin2018explaining,
  title={Explaining explanations: An overview of interpretability of machine learning},
  author={Gilpin, Leilani H and Bau, David and Yuan, Ben Z and Bajwa, Ayesha and Specter, Michael and Kagal, Lalana},
  booktitle={2018 IEEE 5th International Conference on data science and advanced analytics (DSAA)},
  pages={80--89},
  year={2018},
  organization={IEEE}
}

@article{murdoch2019definitions,
  title={Definitions, methods, and applications in interpretable machine learning},
  author={Murdoch, W James and Singh, Chandan and Kumbier, Karl and Abbasi-Asl, Reza and Yu, Bin},
  journal={Proceedings of the National Academy of Sciences},
  volume={116},
  number={44},
  pages={22071--22080},
  year={2019},
  publisher={National Academy of Sciences}
}

@article{bronstein2017geometric,
  title={Geometric deep learning: going beyond euclidean data},
  author={Bronstein, Michael M and Bruna, Joan and LeCun, Yann and Szlam, Arthur and Vandergheynst, Pierre},
  journal={IEEE Signal Processing Magazine},
  volume={34},
  number={4},
  pages={18--42},
  year={2017},
  publisher={IEEE}
}

@article{bronstein2021geometric,
  title={Geometric deep learning: Grids, groups, graphs, geodesics, and gauges},
  author={Bronstein, Michael M and Bruna, Joan and Cohen, Taco and Veli{\v{c}}kovi{\'c}, Petar},
  journal={arXiv preprint arXiv:2104.13478},
  year={2021}
}

@article{gerken2023geometric,
  title={Geometric deep learning and equivariant neural networks},
  author={Gerken, Jan E and Aronsson, Jimmy and Carlsson, Oscar and Linander, Hampus and Ohlsson, Fredrik and Petersson, Christoffer and Persson, Daniel},
  journal={Artificial Intelligence Review},
  volume={56},
  number={12},
  pages={14605--14662},
  year={2023},
  publisher={Springer}
}

@article{meyer2023exploiting,
  title={Exploiting symmetry in variational quantum machine learning},
  author={Meyer, Johannes Jakob and Mularski, Marian and Gil-Fuster, Elies and Mele, Antonio Anna and Arzani, Francesco and Wilms, Alissa and Eisert, Jens},
  journal={PRX quantum},
  volume={4},
  number={1},
  pages={010328},
  year={2023},
  publisher={APS}
}

@article{larocca2022group,
  title={Group-invariant quantum machine learning},
  author={Larocca, Mart{\'\i}n and Sauvage, Fr{\'e}d{\'e}ric and Sbahi, Faris M and Verdon, Guillaume and Coles, Patrick J and Cerezo, Marco},
  journal={PRX quantum},
  volume={3},
  number={3},
  pages={030341},
  year={2022},
  publisher={APS}
}

@article{bradshaw2025learning,
  title={Learning equivariant maps with variational quantum circuits},
  author={Bradshaw, Zachary P and Evans, Ethan N and Cook, Matthew and LaBorde, Margarite L},
  journal={Physical Review Applied},
  volume={23},
  number={4},
  pages={044007},
  year={2025},
  publisher={APS}
}
\end{document}